\documentclass[letterpaper, 11pt, thm-restate]{article}
\usepackage{amsthm}
\usepackage{amsmath}
\usepackage{amssymb}
\usepackage{amsfonts}
\usepackage{mathtools}
\usepackage{array}
\usepackage{xcolor}
\usepackage{leftindex}
\usepackage{thm-restate}
\usepackage{changepage}
\usepackage{caption}
\usepackage[margin=1in]{geometry}
\usepackage{enumitem}
\usepackage[ruled,noend,noline]{algorithm2e}
\usepackage[colorlinks = true]{hyperref}
\hypersetup{
    linkcolor=black,
    citecolor=blue,
    urlcolor=blue}
\hypersetup{hypertexnames=false}
    
\usepackage{titlesec}
\titleformat{\subsection}[runin]
       {\normalfont\bfseries}
       {\thesubsection}
       {0.5em}
       {}
       [.]
\makeatletter
\DeclareFontFamily{U}{mathx}{\hyphenchar\font45}
\DeclareFontShape{U}{mathx}{m}{n}{
      <5> <6> <7> <8> <9> <10>
      <10.95> <12> <14.4> <17.28> <20.74> <24.88>
      mathx10
      }{}
\DeclareSymbolFont{mathx}{U}{mathx}{m}{n}
\DeclareMathAccent{\widecheck}{0}{mathx}{"71}
\makeatother

\numberwithin{figure}{section}
\newtheorem{theorem}{Theorem}[section]
\newtheorem{lemma}[theorem]{Lemma}
\newtheorem{proposition}[theorem]{Proposition}
\newtheorem{corollary}[theorem]{Corollary}
\theoremstyle{definition}
\newtheorem{definition}[theorem]{Definition}
\theoremstyle{definition}

\theoremstyle{definition}

\theoremstyle{definition}
\newtheorem{observation}[theorem]{Observation}

\newcommand{\Z}{\mathbb{Z}}

\newcommand{\N}{\mathbb{N}}

\newcommand{\K}{\mathbb{K}}

\newcommand{\gp}{\operatorname{gp}}
\newcommand{\sgn}{\operatorname{sgn}}

\newcommand{\GNn}{G^{(\mathbb{N}^n)}}
\newcommand{\GN}[1]{G^{(\mathbb{N}^{#1})}}
\newcommand{\GZn}{G^{(\mathbb{Z}^n)}}

\newcommand{\LT}{\operatorname{LT}}
\newcommand{\LC}{\operatorname{LC}}

\newcommand{\supp}{\operatorname{supp}}
\newcommand{\multideg}{\operatorname{multideg}}

\newcommand{\lex}{\operatorname{lex}}

\newcommand{\grevlex}{\operatorname{grevlex}}

\newcommand{\gen}[1]{\left\langle #1 \right\rangle}
\newcommand{\li}[1]{\leftindex^{#1}}

\begin{document}

\title{Standard bases for shift-stable groups and Subgroup Membership in wreath products}

\author{Ruiwen Dong\footnote{Magdalen College, University of Oxford, United Kingdom, email: ruiwen.dong@magd.ox.ac.uk}}

\date{}

\maketitle

\begin{abstract}
    We develop a notion of \textit{standard bases} for subgroups of the restricted direct product $G^{(\mathbb{N}^n)}$ that are stable under translation by $\mathbb{N}^n$, where $G$ is an arbitrary finite group.
    We construct an algorithm that computes standard bases for such subgroups and use them to solve several algorithmic problems, including membership, saturation, and variable elimination.
    Our approach is inspired by Buchberger's algorithm and the theory of Gr\"{o}bner bases for ideals in polynomial rings.
    Building on the standard bases and our solutions to the algorithmic problems above, we prove that Subgroup Membership is decidable in wreath products $G \wr \mathbb{Z}^n$ for finite $G$ and $n \in \mathbb{N}$.
\end{abstract}

%\thispagestyle{empty}
%\newpage
%\setcounter{page}{1}

\section{Introduction}

\subsection{Subgroup Membership and connections to algorithmic algebra}\label{subsec:intro}
The original motivation of this paper comes from computational group theory.
In particular, we are interested in the \emph{Subgroup Membership} problem: given elements $h_1, \ldots, h_k, h$ in a group $H$, decide whether $h$ is in the subgroup generated by $h_1, \ldots, h_k$.
The problem was first explicitly formulated by Mikhailova~\cite{mikhailova1958occurrence}, who showed its undecidability in the direct product of two free groups.
It has since been extensively studied, see~\cite{lohrey2024membership} for a comprehensive survey.
%It has been studied extensively, see~\cite{lohrey2024membership} for a comprehensive survey, and~\cite{babai1996multiplicative, friedl2016membership, potapov2017decidability, Gray2020, bishop2026et0l} for some examples of recent decidability results and related problems.

For a large class of infinite groups, notably solvable groups, membership problems share a deep connection with algorithmic problems in rings and algebras.
Romanovskii~\cite{romanovskii1974some} famously proved decidability of Subgroup Membership in \emph{metabelian groups} by reducing it to the \emph{Submodule Membership} problem over integer polynomial rings: a generalization of the well-known \emph{Ideal Membership} problem.
%A group $H$ is called metabelian if it admits an abelian normal subgroup $A$ such that $H/A$ is abelian.
Romanovskii's solution provided an effective version of \emph{Hilbert's basis theorem}, and from a modern viewpoint it can be summarized as computing a \emph{Gr\"{o}bner basis}.
%This decidability result was later extended to abelian-by-nilpotent groups~\cite{romanovskii1980occurrence}.

Recall that a group $H$ is called metabelian if it admits an abelian normal subgroup $A$ such that $H/A$ is abelian.
A key step of Romanovskii's algorithm is to embed metabelian groups as quotients of the \emph{wreath product} $G \wr \Z^n$, where $G$ is an abelian group.
%Wreath products serve as fundamental building blocks for more general groups.
%, serving as building blocks for groups with significant algebraic and geometric properties~\cite{magnus1939theorem, KaimanovichVershik1983}.
%They also have important applications in graphs, automata, and coding theory~\cite{krohn1965algebraic, Babai1995Automorphism, giudici1999completely}.
For groups $G$ and $F$, their (restricted) wreath product $G \wr F$ is the semidirect product $G^{(F)} \rtimes F$, where $G^{(F)}$ is the restricted direct product of $G$ indexed by $F$, on which $F$ acts by translation.
See Sections~\ref{subsec:intro_inv} and~\ref{subsec:intro_wreath} for exact definitions.

When $G$ is abelian, $G^{(F)}$ admits the natural structure of a module over the group ring $\Z[F]$.
For $F = \Z^n$, this reduces Subgroup Membership in $G \wr F$ to Submodule Membership over the Laurent polynomial ring $\Z[F] = \Z[X_1^{\pm}, \ldots, X_n^{\pm}]$, which is decidable by computing a Gr\"{o}bner basis~\cite{romanovskii1974some}.
This approach can be extended beyond the case of $F = \Z^n$.
Indeed, there is a substantial body of work on algorithms for $\Z[F]$-modules with non-commutative $F$, including when $F$ is nilpotent~\cite{romanovskii1980occurrence, BaumslagCannonitoMiller1981}, polycyclic~\cite{kandri1990non}, or free~\cite{heyworth2001one}: these are collectively known under the term \emph{non-commutative Gr\"{o}bner bases}.

When $G$ is non-abelian, the group $G^{(F)}$ loses its $\Z[F]$-module structure, as it no longer admits abelian addition.
In this case, Subgroup Membership in $G \wr F$ or $G \wr \Z^n$ falls outside the framework of the group--module correspondence.
This requires us to develop new tools beyond the established methods in algorithmic algebra.
The natural first step is to consider finite, non-abelian $G$.
This leads to the two main purposes of this paper:
\begin{enumerate}[wide, nosep, label=\arabic*.]
    \item We develop a notion of \emph{standard bases} for \emph{shift-stable subgroups} of $\GNn$, where $G$ is an arbitrary finite group. 
    %This is inspired by the theory of Gr\"{o}bner bases for ideals in polynomial rings.
    We construct an algorithm that computes standard bases for shift-stable subgroups, allowing us to solve essential problems such as membership, saturation, and elimination.
    
    \item Using the standard bases and our algorithm, we prove that Subgroup Membership is decidable in wreath products $G \wr \Z^n$, where $G$ is finite and $n \in \N$.
    This extends the special case of \mbox{$n=1$} resolved by Lohrey, Steinberg and Zetzsche~\cite{lohrey2015rational}.
    %, and answers the open problem~\cite[Problem~6.2.1]{potthast2020submonoid}.
    Groups of the form $G \wr \Z^n$ are often referred to as \emph{lamplighter groups} and are widely studied for their significant algorithmic and geometric properties~\cite{eskin2013coarse, DBLP:conf/icalp/FigeliusGLZ20, genevois2024asymptotic, dong2025lamplighter, bodart2025finite}.
    We point out that our decidability result cannot be extended to the more general \emph{Rational Subset Membership} problem, which Lohrey and Steinberg proved undecidable in $G \wr \Z^n$ for non-trivial $G$ and $n \geq 2$~\cite{lohrey2011tilings}.
\end{enumerate}

\subsection{Shift-stable subgroups of $\GNn$ and standard bases}\label{subsec:intro_inv}
We adopt the convention that $0 \in \N$.
Let $G$ be a group and denote by $e$ its neutral element.
For a map $f \colon \N^n \rightarrow G$, its \emph{support} is defined as the set
\[
\supp(f) \coloneqq \{\alpha \in \N^n \mid f(\alpha) \neq e\}.
\]
The \emph{restricted direct product} $\GNn$ is the set of all finitely supported maps $f \colon \N^n \rightarrow G$:
\[
\GNn \coloneqq \{f \colon \N^n \rightarrow G \mid \supp(f) \text{ is finite}\}.
\]
It forms a group under pointwise multiplication: for $f_1, f_2 \in \GNn$, their product is the map 
\[
f_1 f_2 \colon \N^n \rightarrow G, \; \alpha \mapsto f_1(\alpha) f_2(\alpha).
\]
The group $\GNn$ is equipped with a \emph{shift} action: for $\beta \in \N^n$ and $f \in \GNn$, let $\li{\beta} f$ denote the element of $\GNn$ defined by 
\[
\li{\beta} f(\alpha) = 
\begin{cases}
    f(\alpha - \beta) \quad & \text{ if } \alpha - \beta \in \N^n, \\
    e \quad & \text{ otherwise}.
\end{cases}
\]
Then $\beta \in \N^n$ induces a group homomorphism 
\[
\li{\beta}(\cdot) \colon \GNn \rightarrow \GNn,\quad f \mapsto \li{\beta}f.
\]

\begin{figure}[hb!]
    \captionsetup{belowskip=-7pt}
    \centering
    \begin{minipage}[t]{.27\textwidth}
        \centering
        \includegraphics[width=0.8\textwidth,height=0.8\textheight,keepaspectratio, trim={6.25cm 1.1cm 5.8cm 1cm},clip]{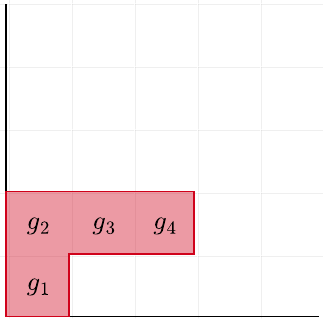}
        \caption{The tile $f_1$}
        \label{fig:visualf}
    \end{minipage}
    \hfill
    \begin{minipage}[t]{.27\textwidth}
        \centering
        \includegraphics[width=0.8\textwidth,height=0.8\textheight,keepaspectratio, trim={6.25cm 1.1cm 5.8cm 1cm},clip]{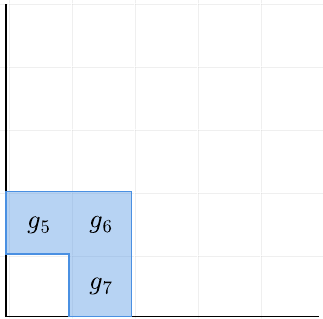}
        \caption{The tile $f_2$}
        \label{fig:visualg}
    \end{minipage}
    \hfill
    \begin{minipage}[t]{0.27\textwidth}
        \centering
        \includegraphics[width=0.8\textwidth,height=0.8\textheight,keepaspectratio, trim={6.25cm 1.1cm 5.8cm 1cm},clip]{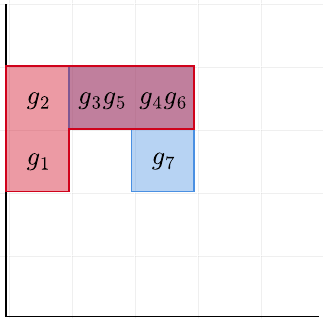}
        \caption{The product $\li{(0, 2)}f_1 \cdot \li{(1, 2)}f_2$.}
        \label{fig:visualfg}
    \end{minipage}
\end{figure}

For visualization purposes, it may be useful to think of an element $f \in \GNn$ as a ``tile'': it consists of a finite set $\supp(f) \subset \N^n$ of ``cells'', each of which contains a non-trivial element of $G$.
The shift $f \mapsto \li{\beta}f$ translates the tile by a vector $\beta \in \N^n$.
The product $f_1 f_2 \cdots f_k$ corresponds to superimposing the tiles $f_1, f_2, \ldots, f_k$, with $f_1$ as the top layer, and multiplying the entries in each cell from top to bottom.
See Figures~\ref{fig:visualf}, \ref{fig:visualg}, \ref{fig:visualfg} for illustration.

A subgroup $I \leq \GNn$ is called \emph{shift-stable} or \emph{stable}, if 
\[
\li{\beta} I \subseteq I
\]
for all $\beta \in \N^n$.
%We say a stable subgroup $I$ is \emph{generated} by a set $S \subseteq \GNn$, if $I$ is the smallest stable subgroup containing $S$; in this case we denote $I = \gen{S}$.
The stable subgroup generated by a set $S \subseteq \GNn$ is the group generated by the elements $\li{\beta}s$, where $\beta \in \N^n, s \in S$.
This is the smallest stable subgroup containing $S$, and we denote it by $\gen{S}$.
A first natural problem one encounters is the \emph{membership problem}: given $f \in \GNn$ and $S \subset \GNn$, decide whether $f \in \gen{S}$.

When $G$ is the abelian group $(\Z, +)$, one can identify an element $f \in \GNn$ with the polynomial 
\[
\sum_{(\alpha_1, \ldots, \alpha_n) \in \N^n} f(\alpha_1, \ldots, \alpha_n) X_1^{\alpha_1} \cdots X_n^{\alpha_n} \in \Z[X_1, \ldots, X_n].
\]
Under this identification, $\GNn$ is simply $(\Z[X_1, \ldots, X_n], +)$: the additive group of a polynomial ring.
The shift $f \mapsto \li{\beta}f$ corresponds to multiplying $f$ by the monomial $X_1^{\beta_1} \cdots X_n^{\beta_n}$.
In this case, shift-stable subgroups of $\GNn$ are exactly \emph{ideals} of $\Z[X_1, \ldots, X_n]$.
Indeed, an additive subgroup $I \leq \Z[X_1, \ldots, X_n]$ is shift-stable if and only if it is stable under multiplication by monomials, and this is equivalent to $I$ being stable under multiplication by polynomials, i.e., $I$ is an ideal.
%Computation of ideals in polynomial rings is a staple of algorithmic algebra.
One of the most influential tools for computation in ideals is \emph{Gr\"{o}bner bases} (also known as \emph{standard bases}), due to the seminal works of Hironaka~\cite{Hironaka1964} and Buchberger~\cite{Buchberger1965}.
A Gr\"{o}bner basis of an ideal can be effectively computed from its generators by the well-known \emph{Buchberger's algorithm}~\cite{Buchberger1965}.
This allows one to solve essential algorithmic problems such as the \emph{membership problem}, \emph{ideal saturation}, and \emph{variable elimination}.
See~\cite{CoxLittleOShea2015} for a modern treatment of the subject.

When $G$ is a non-abelian group, traditional tools like Gr\"{o}bner bases no longer apply.
Current generalizations of Gr\"{o}bner bases focus on relaxing commutativity for the \emph{multiplicative} structure of the ring $R[X_1, \ldots, X_n]$~\cite{Mora1994, Bokut2020GrobnerShirshov}.
%These generalizations do not relax commutativity of the \emph{additive} structure.
In contrast, shift-stable subgroups of $\GNn$ replace the \emph{additive} structure of $R$ by a non-abelian $G$.
%This lands outside the scope of known generalizations.

To adapt the classical ideas to the setting of stable subgroups, in Section~\ref{subsec:basis_def} we define a notion of standard bases for shift-stable subgroups of $\GNn$, where $G$ is an arbitrary finite group.
%Our notion generalizes the classical Gr\"{o}bner basis for commutative $G$.
In Section~\ref{subsec:basis_comp} we construct an algorithm that, given a finite generating set of a shift-stable subgroup $I \leq \GNn$, computes a standard basis for $I$.
%Implicitly, this gives an effective Hilbert basis theorem for shift-stable subgroups of $\GNn$ (i.e., all stable subgroups are finitely generated).
The main challenge in constructing the algorithm is to establish a termination criterion.
In the classical setting of ideals, \emph{Buchberger's criterion} ensures that an algorithm can terminate once all \emph{S-polynomials} reduce to zero.
The stable subgroup setting is much more subtle, as the non-commutativity of $G$ hinders the usage of this criterion.
To overcome this, we develop a notion of \emph{S-tiles}, which takes into account non-trivial conjugations.
The key technical contribution is Theorem~\ref{thm:criterion_tile}, which establishes a termination criterion using S-tiles.
%This notion is the key to our algorithm.
The rest of Section~\ref{sec:inv} uses standard bases to solve a list of algorithmic problems.
This is summarized partially as the following theorem:

\begin{theorem}[Section~\ref{sec:inv}]\label{thm:informal}
    %Let $G$ be a finite group.
    There is an algorithm that, given a finite group $G$ and a finite set of generators of a shift-stable subgroup $I \leq \GNn$, computes a standard basis of $I$ (Theorem~\ref{thm:compute_tile_basis}).
    Furthermore, given the generators of $I$, there are algorithms for the following problems:
    \begin{enumerate}[nosep, label=\roman*.]
        \item (Membership problem, Theorem~\ref{thm:membership}) given $f \in \GNn$, decide whether $f \in I$.
        \item (Subgroup saturation, Theorem~\ref{thm:Bayer_tile}) for $\varepsilon_n = (0, 0, \ldots, 1) \in \N^n$, compute a finite set of generators for the shift-stable subgroup
        \[
        I \colon \varepsilon_n^{\infty} \coloneqq \left\{f \in G^{(\N^n)} \;\middle|\; \exists t \in \N \text{ such that } \li{t\varepsilon_n}f \in I \right\}.
        \]
        \item (Variable elimination, Theorem~\ref{thm:elimination_tile}) given $d \geq 1$, compute a finite set of generators for the $(\N^{n-1}\text{-})$shift-stable subgroup
        \[
        I_{[0, d-1] \times \N^{n-1}} \coloneqq \{f \in I \mid \supp(f) \subseteq [0, d-1] \times \N^{n-1}\} \leq (G^d)^{(\N^{n-1})}.
        \]
    \end{enumerate}
\end{theorem}

Here, the finite group $G$ may be given by any effective representation, for example its multiplication table.
%In Section~\ref{subsec:nested} we will also define and solve the \emph{nested subgroup membership} problem, which will be crucial for later applications in $G \wr \Z^n$.
We mention that a similarly named concept of \emph{group shifts} has been studied in the context of symbolic dynamics~\cite{kitchens1989automorphisms, beaur2024effective}.
These are shift-invariant subgroups of the unrestricted direct product $G^{\Z^d}$ which are \emph{topologically closed}.
However, their algorithmic framework differs substantially from stable subgroups of $\GNn$ or $\GZn$.
%However, topological closed-ness means that group shifts can be represented by \emph{forbidden patterns}, hence their algorithmic framework differs substantially from stable subgroups of $\GNn$ or $\GZn$.

\subsection{Subgroup Membership in wreath products}\label{subsec:intro_wreath}
In the second part of this paper, we apply the tools from Theorem~\ref{thm:informal} to decide Subgroup Membership in wreath products $G \wr \Z^n$ for finite $G$.

For a group $G$, the restricted direct product $\GZn$ is defined similarly to $\GNn$: its elements are finitely supported maps from $\Z^n$ to $G$.
For $\beta \in \Z^n$ and $f \in \GZn$, the element $\li{\beta} f \in \GZn$ is defined as
$
\li{\beta} f (\alpha) = f(\alpha - \beta),\; \alpha \in \Z^n
$.
The wreath product $G \wr \Z^n$ is the group whose elements are pairs $(f, \alpha)$ with $f \in \GZn, \alpha \in \Z^n$, and where multiplication is defined by
\[
(f, \alpha) \cdot (g, \beta) = (f \cdot \li{\alpha}g,\; \alpha + \beta).
\]
%Let $0$ denote the neutral element of $\GZn$, then 
The neutral element is $(e^{\Z^n}, 0)$.
See Figures~\ref{fig:wreathf}-\ref{fig:wreathfg} for an illustration: one can visualize $(f, \alpha) \in G \wr \Z^n$ as a tile $f \in \GZn$ together with a vector $\alpha \in \Z^n$.
Left-multiplication by $(f, \alpha)$ corresponds to translating by $\alpha$ and applying the tile $f$ on top.
%See Figures~\ref{fig:wreathf}, \ref{fig:wreathg}, \ref{fig:wreathfg} for illustration.
In Section~\ref{sec:membership} of this paper, we show:

\begin{restatable}{theorem}{thmgroup}\label{thm:group}
    Subgroup Membership is (uniformly) decidable in $G \wr \Z^n$ for finite $G$ and $n \in \N$.
\end{restatable}
Formally, the algorithm we give takes as input the finite group $G$, the number $n$, elements $h_1, \ldots, h_k, h \in G \wr \Z^n$, and decides whether $h$ is in the subgroup generated by $h_1, \ldots, h_k$.

By a reduction of Shafrir~\cite[Theorem~4.3.4]{potthast2020submonoid}, Theorem~\ref{thm:group} also implies that \emph{Submonoid Membership} is decidable in $G \wr \Z^2$ for finite $G$.
Whether Submonoid Membership is decidable in $G \wr \Z^n$ for finite $G$ and $n \geq 3$ remains an open problem.

It is worth noting that a result of Lohrey, Steinberg and Zetzsche~\cite{lohrey2015rational} showed that Subgroup Membership (and Rational Subset Membership) is decidable in wreath products $G \wr F$, where $G$ is finite and $F$ is \emph{virtually free}.
Their algorithm is based on a refinement of the well-quasi-order of subsequences~\cite{higman1952ordering}.
It would be interesting to investigate in a broader context how this relates to the well-quasi-order of ideals and subgroups in our work.
See Table~\ref{tbl:stateofart} for a comparison of our result with some related work on membership problems in wreath products.

\begin{table}[h!]
%\setstretch{1.5}
\newcolumntype{C}[1]{>{\centering\arraybackslash}m{#1}}
\centering
\begin{tabular}{ | C{1.2cm} | C{1.5cm}||C{2.2cm}| C{3cm}|| C{2.8cm} | C{3cm} |} \hline
    $G$ & $F$ & Subgroup Membership & (main argument) & Rational Subset Membership & (main argument) \\
  \hline
  \hline
  abelian & $\Z^n$ & decidable \cite{romanovskii1974some} & Gr\"{o}bner bases & undecidable \cite{lohrey2015rational} $(\Z \wr \Z)$ & two-counter automata \\
  %\hline
  %abelian & nilpotent & decidable \cite{romanovskii1980occurrence} & Non-commutative Gr\"{o}bner bases & undecidable & subsumed by $\uparrow$ \\
  \hline
  finite & $\Z^n$ & \textbf{decidable} \textbf{(Section~\ref{sec:membership})} & \textbf{standard bases} \textbf{(Section~\ref{sec:inv})} & undecidable \cite{lohrey2011tilings} $(n \geq 2)$ & 2D-tiling problem\\
  %\hline
  %$\Z^m$ & free & $?$ & & undecidable & subsumed by above \\
  \hline
  finite & free & decidable & subsumed by $\rightarrow$ & decidable \cite{lohrey2015rational} & well-quasi-order of subsequences \\
  \hline
\end{tabular}
%\vspace{0.5cm}
\caption{\label{tbl:stateofart} Selected decidability results for wreath products $G \wr F$, and their main arguments}
\end{table}
\vspace{-0.2cm}

\section{Preliminaries, Gr\"{o}bner bases, and Buchberger's algorithm}\label{sec:prelim}
In this section we recall the computational basics on ideals in polynomial rings, including the definition of monomial ordering, Gr\"{o}bner basis, and Buchberger's algorithm.
Although these preliminaries are not strictly needed for our subsequent study of shift-stable subgroups, they give an overview of the classical ideas and provide a comparison between known work and our generalization.
The treatment in this section closely follows the monograph~\cite[Chapter~2]{CoxLittleOShea2015}.

\subsection{Gr\"{o}bner bases}\label{subsec:prelim_GB}
%We adopt the convention that $0 \in \N$.
An \emph{ideal} of a commutative ring $R$ is a subset $I \subseteq R$ that satisfies: (i) $x, y \in I \implies x+y \in I$, (ii) $x \in I \implies -x \in I$, and (iii) $r \in R, x \in I \implies rx \in I$.
For a set of elements $S \subseteq R$, denote by $\gen{S}$ the ideal generated by $S$: this is the set of finite sums $\sum_{s \in S} r_s s$ with $r_s \in R$.

\begin{definition}[Monomial order]\label{def:order}
    Let $n \geq 0$.
    A \emph{monomial order} on $\N^n$ is a relation $<$ on the set $\N^n$, satisfying:
    \begin{enumerate}[nosep, label=(\roman*)]
        \item $<$ is a total order on $\N^n$.
        \item If $\alpha < \beta$ and $\gamma \in \N^n$, then $\alpha + \gamma < \beta + \gamma$.
        \item Every non-empty subset of $\N^n$ has a smallest element under $<$ (i.e., $<$ is a well-order).
    \end{enumerate}
\end{definition}
We write $\alpha \leq \beta$ if $\alpha < \beta$ or $\alpha = \beta$.
For a vector $\alpha \in \N^n$, we will often write $\alpha_i$ as its $i$-th coordinate, i.e., $\alpha = (\alpha_1, \ldots, \alpha_n)$.
A typical monomial order is the \emph{lexicographic order} $<_{\lex}$, where we define $\alpha <_{\lex} \beta$ if and only if $
\alpha_1 = \beta_1, \ldots, \alpha_{k-1} = \beta_{k-1}, \alpha_k < \beta_k,
$ for some $1 \leq k \leq n$.
For example, $(0, 0) <_{\lex} (0, 1) <_{\lex} (1,0) <_{\lex} (1, 1)$, i.e., the left-most entry is the most significant.

Let $\K$ be a field and $\K[X_1, \ldots, X_n]$ be the ring of polynomials with $n$ variables and coefficients in $\K$.
For $\alpha \in \N^n$ we write $X^{\alpha}$ for the monomial $X_1^{\alpha_1} \cdots X_n^{\alpha_n}$.
Let $f = \sum_{\alpha \in \N^n} c_{\alpha} X^{\alpha}$ be a polynomial in $\K[X_1, \ldots, X_n]$, and let $<$ be a monomial order, we define:
\begin{enumerate}[nosep, label=(\roman*)]
    \item The \emph{support} of $f$ is the (finite) set 
    \[
    \supp(f) \coloneqq \{\alpha \in \N^n \mid c_{\alpha} \neq 0\} \subset \N^n.
    \]
    \item The \emph{multidegree} of $f$ is the maximal element in $\supp(f)$ with respect to the order $<$:
    \[
    \multideg(f) \coloneqq \max_{<} (\supp(f)) \in \N^n.
    \]
    \item The \emph{leading coefficient} of $f$ is
    $
    \LC(f) \coloneqq c_{\multideg(f)} \in \K.
    $
    \item The \emph{leading term} of $f$ is
    $
    \LT(f) \coloneqq \LC(f) \cdot X^{\multideg(f)} \in \K[X_1, \ldots, X_n].
    $
\end{enumerate}
\smallskip

A special kind of ideals that play an important role are \emph{monomial ideals}:
\begin{definition}[Monomial ideal]
    A \emph{monomial} is a polynomial of the form $X^{\alpha}, \alpha \in \N^n$.
    An ideal $I \subseteq \K[X_1, \ldots, X_n]$ is called a \emph{monomial ideal} if it is generated by a (possibly infinite) set of monomials.
\end{definition}
Let $I$ be a monomial ideal and denote $A \coloneqq \{\alpha \in \N^n \mid X^{\alpha} \in I\}$.
Then $A + \N^n \subseteq A$, see Figure~\ref{fig:monomial} for an illustration.
Conversely, let $A \subseteq \N^n$ be a set satisfying $A + \N^n \subseteq A$, then $I \coloneqq \big\{\sum_{\alpha \in A} c_{\alpha} X^\alpha \;\big|\; c_\alpha = 0 \text{ for all but finitely many } \alpha \big\}$ is a monomial ideal, generated by $X^{\alpha}, \alpha \in A$.
Hence there is a one-to-one correspondence between monomial ideals and sets satisfying $A + \N^n \subseteq A$.
%A monomial ideal is always finitely generated:

\begin{lemma}[{Dickson's lemma~\cite[Chapter~2.4]{CoxLittleOShea2015}}]\label{lem:Dickson}
    Every set $A \subseteq \N^n$ satisfying $A + \N^n \subseteq A$ can be written as $A = \{\alpha_1, \ldots, \alpha_k\} + \N^n$ for some finite set $\{\alpha_1, \ldots, \alpha_k\} \subseteq A$.
    Equivalently, every monomial ideal is generated by finitely many monomials $X^{\alpha_1}, \ldots, X^{\alpha_k}$.
\end{lemma}

The following reduction algorithm is a generalization of the Euclidean division algorithm:

\begin{lemma}[{Lead-reduction, see~\cite[Chapter~2.3]{CoxLittleOShea2015}}]\label{lem:reduction_ideal}
    Let $<$ be a monomial order on $\N^n$, and let $F = \{f_1, \ldots, f_k\}$ be a finite set of elements in $\K[X_1, \ldots, X_n]$.
    Then every $f \in \K[X_1, \ldots, X_n]$ can be written (by an algorithm) as
    \[
    f = q_1 f_1 + \cdots + q_k f_k + r,
    \]
    where $q_i, r \in \K[X_1, \ldots, X_n]$, and
    \begin{enumerate}[nosep, label=(\roman*)]
        \item $q_i f_i$ is either zero or $\multideg(q_i f_i) \leq \multideg(f)$.
        \item $r$ is either zero or $\LT(r) \notin \gen{\LT(f_1), \ldots, \LT(f_k)}$.
    \end{enumerate}
\end{lemma}

In Lemma~\ref{lem:reduction_ideal}, the element $r$ is called a \emph{reduction of $f$ by $F$}. 
In particular, $f -r \in \gen{f_1, \ldots, f_k}$.
In general, the reduction $r$ is not unique.
For example, we reduce $f = X Y^2 - X$ by $F = \{f_1 = X Y - 1,\; f_2 = Y^2 - 1\}$ in the lexicographic order.
We can write $f = Yf_1 + 0f_2 + (Y - X)$, and $\LT(Y - X) = -X \notin \gen{\LT(f_1) = X Y,\; \LT(f_2) = Y^2}$, so $Y-X$ is a reduction of $f$ by $F$.
But we can also write $f = Xf_1 + 0 f_2 + 0$, so $0$ is also a reduction of $f$ by $F$.

Note that the ideal $\gen{\LT(f_1), \ldots, \LT(f_k)}$ in Lemma~\ref{lem:reduction_ideal} is a monomial ideal generated by $\LT(f_i)/\LC(f_i), i = 1, \ldots, k$.
For an ideal $I$, define the set
$
\LT(I) \coloneqq \{\LT(f) \mid f \in I\}
$,
and let $\gen{\LT(I)}$ be the ideal it generates.
Combining Dickson's lemma with the reduction algorithm yields:

\begin{corollary}[Hilbert's basis theorem]\label{cor:Hilbert}
    Every ideal $I$ of $\K[X_1, \ldots, X_n]$ is finitely generated.
\end{corollary}
\begin{proof}
    Fix any monomial order.
    Applying Dickson's Lemma to the monomial ideal $\gen{\LT(I)}$, we obtain $f_1, \ldots, f_k \in I$ such that $\gen{\LT(I)} = \gen{\LT(f_1), \ldots, \LT(f_k)}$.
    We show that $I = \gen{f_1, \ldots, f_k}$.
    For any $f \in I$, let $r \in I$ be its reduction by $\{f_1, \ldots, f_k\}$.
    Then $\LT(r) \in \LT(I)$ $\subseteq \gen{\LT(f_1), \ldots, \LT(f_k)}$, so $r$ must be zero by definition of reduction, hence $f \in \gen{f_1, \ldots, f_k}$.
\end{proof}

The choice of $f_1, \ldots, f_k$ in the above proof motivates the definition of a \emph{Gr\"{o}bner basis}:

\begin{definition}[{Gr\"{o}bner basis~\cite[Chapter~2.5]{CoxLittleOShea2015}}]\label{def:Groebner}
    Fix a monomial order on $\N^n$.
    A finite subset $B = \{b_1, \ldots, b_k\}$ of an ideal $I \subseteq \K[X_1, \ldots, X_n]$ is called a \emph{Gr\"{o}bner basis}, if
    \begin{equation*}
    \gen{\LT(I)} = \gen{\LT(b_1), \ldots, \LT(b_k)}.
    \end{equation*}
\end{definition}

When a Gr\"{o}bner basis of an ideal is known, we can use it to decide membership in the ideal:

\begin{theorem}[Ideal membership]\label{thm:reduction_ideal}
    Let $B = \{b_1, \ldots, b_k\}$ be a Gr\"{o}bner basis of an ideal $I$ and $f \in \K[X_1, \ldots, X_n]$. 
    If $f \in I$ then all reductions of $f$ by $B$ are zero.
    Conversely, if any reduction of $f$ by $B$ is zero then $f \in I$.
\end{theorem}
\begin{proof}
    Let $f \in I$ and let $r$ be a reduction of $f$ by $B$.
    Then $r \in I$, so $\LT(r) \in \gen{\LT(b_1), \ldots, \LT(b_k)}$ by definition of Gr\"{o}bner basis.
    Hence $r$ is zero by definition of reduction.
    Conversely if any reduction of $f$ by $B$ is zero then $f \in \gen{B} \subseteq I$.
\end{proof}

Note that Definition~\ref{def:Groebner} does not require the subset $B \subseteq I$ to \emph{generate} $I$, but it automatically follows from Theorem~\ref{thm:reduction_ideal} that a Gr\"{o}bner basis of $I$ generates $I$.

\subsection{Buchberger's algorithm}\label{subsec:prelim_Buchberger}
Thus far, Gr\"{o}bner bases have been defined without providing a way to compute them.
The classical method of computing a Gr\"{o}bner basis is \emph{Buchberger's algorithm}.
The general idea is to start with a set of generators $B$ of the ideal $I$, then repeatedly adjoin new elements to $B$ until it becomes a Gr\"{o}bner basis.
The adjoined elements are reductions of \emph{S-polynomials}.
%, and the algorithm terminates when all new S-polynomials reduce to zero.

Let $f, g$ be two polynomials in $\K[X_1, \ldots, X_n]$, and denote $\alpha \coloneqq \multideg(f), \beta \coloneqq \multideg(g)$.
Let $\gamma \coloneqq (\gamma_1, \ldots, \gamma_n)$ where $\gamma_i = \max(\alpha_i, \beta_i)$, so $X^{\gamma}$ is the least common multiple of $X^{\alpha}$ and $X^{\beta}$.
The \emph{S-polynomial} of $f$ and $g$ is defined as
\[
S(f, g) \coloneqq X^{\gamma - \alpha} \cdot \frac{f}{\LC(f)} - X^{\gamma - \beta} \cdot \frac{g}{\LC(g)}.
\]
For example, in the lexicographic monomial order, if $f = X_1^3 X_2^5 + X_2^5 + 1$ and $g = X_1^4 X_2^2 + 1$, then $\alpha = (3, 5)$, $\beta = (4, 2)$, $\gamma = (4, 5)$, and $S(f, g) = X_1 \cdot (X_1^3 X_2^5 + X_2^5 + 1) - X_2^3 \cdot (X_1^4 X_2^2 + 1)$.
Intuitively, the S-polynomial $S(f,g)$ describes how the leading terms of $f$ and $g$ can ``cancel out'' with each other.
See Figures~\ref{fig:Sf}-\ref{fig:Sfg} for illustration.

\begin{figure}[h!]
    %for 10 page format use b! and move after definition of S(f,g)
    \centering
    \begin{minipage}[t]{.30\textwidth}
        \centering
        \includegraphics[width=0.7\textwidth,height=0.8\textheight,keepaspectratio, trim={6.27cm 1.1cm 5.84cm 1.05cm},clip]{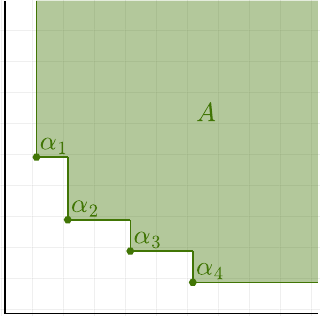}
        \caption{monomial ideal and its generators}
        \label{fig:monomial}
    \end{minipage}
    \hfill
    \begin{minipage}[t]{.21\textwidth}
        \centering
        \includegraphics[width=\textwidth,height=0.8\textheight,keepaspectratio, trim={6.27cm 1.1cm 5.84cm 1.05cm},clip]{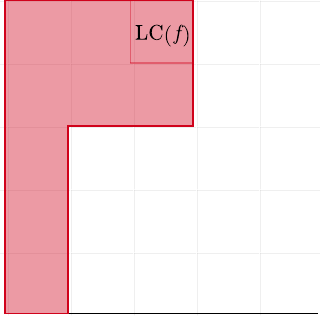}
        \caption{support of a polynomial $f$}
        \label{fig:Sf}
    \end{minipage}
    \hfill
    \begin{minipage}[t]{.21\textwidth}
        \centering
        \includegraphics[width=\textwidth,height=0.8\textheight,keepaspectratio, trim={6.27cm 1.1cm 5.84cm 1.05cm},clip]{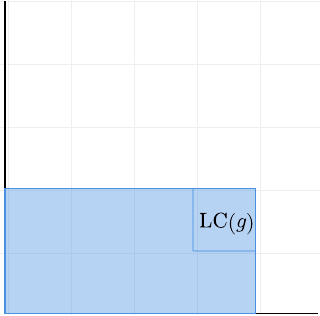}
        \caption{support of a polynomial $g$}
        \label{fig:Sg}
    \end{minipage}
    \hfill
    \begin{minipage}[t]{0.21\textwidth}
        \centering
        \includegraphics[width=\textwidth,height=0.8\textheight,keepaspectratio, trim={6.25cm 1.07cm 5.84cm 1.05cm},clip]{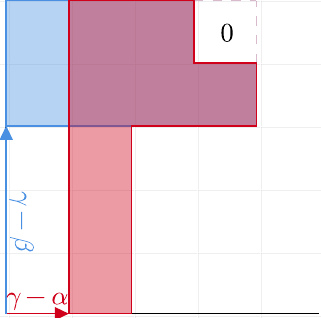}
        \caption{support of $S(f, g)$.}
        \label{fig:Sfg}
    \end{minipage}
\end{figure}

\begin{theorem}[{Buchberger's criterion~\cite[Chapter~2.6]{CoxLittleOShea2015}}]\label{thm:criterion}
    Let $I \subseteq \K[X_1, \ldots, X_n]$ be an ideal and $B = \{b_1, \ldots, b_k\}$ be a generating set of $I$.
    Then $B$ is a Gr\"{o}bner basis if and only if for every $i, j$, some reduction of $S(b_i, b_j)$ by $B$ is zero.
\end{theorem}
\begin{proof}[Sketch of proof]
    If $B$ is a Gr\"{o}bner basis, then the reduction of $S(b_i,b_j) \in I$ by $B$ is trivial by Theorem~\ref{thm:reduction_ideal}.
    Conversely, suppose that every $S(b_i, b_j)$ has a trivial reduction.
    Take any $f \in I$, and we show $\LT(f) \in \gen{\LT(b_1), \ldots, \LT(b_k)}$.
    Write $f = \sum_{i = 1}^k h_i b_i$.
    Denote $\delta \coloneqq \max_{1 \leq i \leq k}(\multideg(h_i b_i))$, then $\multideg(f) \leq \delta$.

    \textbf{Case 1:} $\delta = \multideg(f)$. Then $\LT(f) = \sum_{\multideg(h_i b_i) = \delta} \LT(h_i) \LT(b_i) \in \gen{\LT(b_1), \ldots, \LT(b_k)}$.

    \textbf{Case 2:} $\delta > \multideg(f)$. Then we can write
    \[
    f = \sum_{\multideg(h_ib_i) = \delta} \LT(h_i) b_i + \sum_{\multideg(h_ib_i) = \delta} (h_i - \LT(h_i)) b_i + \sum_{\multideg(h_ib_i) < \delta} h_i b_i.
    \]
    The leading terms of the first sum cancel out, i.e., $\sum_{\multideg(h_ib_i) = \delta} \LT(h_i) \LT(b_i) = 0$.
    It is not hard to show that this means $\sum_{\multideg(h_ib_i) = \delta} \LT(h_i) b_i$ is equal to a linear combination of S-polynomials $\sum_{i, j} X^{\delta - \gamma_{ij}} S(b_i, b_j)$,
    where $\multideg(X^{\delta - \gamma_{ij}} S(b_i, b_j)) < \delta$.
    Since each $S(b_i,b_j)$ reduces to zero by $B$, we have $S(b_i,b_j)=\sum_{\ell = 1}^k a_{\ell} b_{\ell}$ where $\multideg(a_{\ell} b_{\ell}) \leq \multideg(S(b_i,b_j))$.
    Hence, replacing each $S(b_i, b_j)$ by $\sum_{\ell = 1}^k a_{\ell} b_{\ell}$, we can rewrite $\sum_{\multideg(h_ib_i) = \delta} \LT(h_i) b_i = \sum_{i, j} X^{\delta - \gamma_{ij}} S(b_i, b_j)$ as a linear combination $\sum_{i = 1}^k \tilde h_i b_i$ with $\multideg(\tilde h_i b_i) < \delta$.
    This yields an expression $f = \sum_{i = 1}^k h'_i b_i$ where $\max_{1 \leq i \leq k}(\multideg(h'_i b_i))$ is strictly smaller than $\delta$, allowing us to conclude by induction on $\delta$.
\end{proof}

\begin{theorem}[{Buchberger's algorithm~\cite[Chapter~2.7]{CoxLittleOShea2015}}]\label{thm:Buchberger_algorithm}
    There is an algorithm which, given as input a finite set of generators for an ideal $I$, computes a Gr\"{o}bner basis for $I$.
\end{theorem}
\begin{proof}[Sketch of proof]
    The algorithm starts with the generating set $B$ of $I$.
    At each step, it computes a reduction of $S(b_i, b_j)$ by $B$ for all $b_i, b_j \in B$.
    If the reduction is not trivial then append it to $B$.
    At each step, $\gen{\LT(b) \colon b \in B}$ becomes strictly larger.
    By Dickson's lemma, the algorithm must terminate after finitely many steps.
    The resulting set $B$ is a Gr\"{o}bner basis by Theorem~\ref{thm:criterion}.
\end{proof}

\section{Standard bases for shift-stable subgroups}\label{sec:inv}
This section develops a notion of standard bases for shift-stable subgroups of $\GNn$, where $G$ is an arbitrary finite group.
%We keep a parallel labelling to Section~\ref{sec:prelim}.
In Section~\ref{subsec:basis_def} we formulate the definition of a standard basis, in close parallel to Section~\ref{subsec:prelim_GB}.
In Section~\ref{subsec:basis_comp}, we develop an algorithm for computing a standard basis: this is one of our main technical contributions. 
%We maintain the comparison between Section~\ref{subsec:basis_comp} and Section~\ref{subsec:prelim_Buchberger}: due to the non-commutativity of $G$, our algorithm is more involved than Buchberger's algorithm.
In Section~\ref{subsec:saturation} we define the problem of \emph{saturation} for stable subgroups, and give a solution inspired by the work of Bayer~\cite{Bayer1982} on ideal saturation.
In Section~\ref{subsec:eliminiation} we give an algorithm for \emph{variable elimination} for stable subgroups.
In Section~\ref{subsec:nested} we define and solve the \emph{nested subgroup membership} problem: this is another technical contribution of this section.
Nested subgroup membership combines saturation and variable elimination, and will be a crucial component in the resolution of Subgroup Membership in $G \wr \Z^n$ in Section~\ref{sec:membership}.

\subsection{Standard bases}\label{subsec:basis_def}
It turns out that the definitions and results in Section~\ref{subsec:prelim_GB} can be generalized to stable subgroups of $\GNn$ without major obstructions.
Let $G$ be a finite group and let $e$ denote its neutral element.
Recall that $\GNn$ is the group of finitely supported maps $f \colon \N^n \rightarrow G$, equipped with the shift action $f \mapsto \li{\beta} f,\; \beta \in \N^n$.
Owing to the visualization in Figures~\ref{fig:visualf}-\ref{fig:visualfg}, we will call elements of $\GNn$ \emph{tiles}.
Note that the (pointwise) product $f_1 f_2$ of two tiles $f_1, f_2 \in \GNn$ should be compared to the \emph{sum} ``$f_1 + f_2$'' in the polynomial setting, and the (pointwise) inverse $f_1^{-1}$ should be compared to ``$- f_1$'' in the polynomial setting.
In general, $f_1 f_2 \neq f_2 f_1$ when $G$ is non-abelian.
%, unless $\supp(f_1) \cap \supp(f_2) = \emptyset$.

A subgroup $I \leq \GNn$ is (shift-)stable, if $\alpha \in \N^n, f \in I \implies \li{\alpha}f \in I$.
For a set $S \subseteq \GNn$, denote
\[
S^{\pm} \coloneqq S \cup \{s^{-1} \mid s \in S\}.
\]
We denote by $\gen{S}$ the stable subgroup generated by $S$: it is the set of tiles of the form
\[
    \li{\alpha_1}s_1 \cdot \li{\alpha_2}s_2 \cdot \cdots \cdot \li{\alpha_k}s_k \in \GNn,
\]
where $k \in \N, s_1, \ldots, s_k \in S^{\pm}$, and $\alpha_1, \ldots, \alpha_k \in \N^n$.
The notions of multidegree, leading coefficient, and leading term naturally generalize from polynomials to tiles.

\begin{definition}
    Fix a monomial order $<$ on $\N^n$ and let $f \in \GNn$ be a tile.
    \begin{enumerate}[nosep, label=(\roman*)]
        %\item The \emph{support} of $f$ is the finite set $\supp(f) \coloneqq \{\alpha \in \N^n \mid f(\alpha) \neq e\} \subseteq \N^n$.
        \item The \emph{multidegree} of $f$ is the maximal element in $\supp(f)$ with respect to $<$, i.e., $\multideg(f) \coloneqq \max_{<}(\supp(f)) \in \N^n$.
        \item The \emph{leading coefficient} of $f$ is $\LC(f) \coloneqq f(\multideg(f)) \in G$.
        \item The \emph{leading term} of $f$ is the tile $\LT(f) \colon \N^n \rightarrow G$ that sends $\multideg(f)$ to $\LC(f)$, and all other elements of $\N^n$ to $e$.
    \end{enumerate}
    The neutral element of $\GNn$ is the \emph{trivial tile}: it sends every element of $\N^n$ to $e$.
\end{definition}

To generalize the notion of monomials from polynomials to tiles, we need to include the coefficient, since there is no distinguished element of $G$ that serves the purpose of $1 \in \K$.

\begin{definition}[Monomial stable subgroup]
    A tile $f \in \GNn$ is called a \emph{monomial} if $\supp(f)$ is a singleton.
    A stable subgroup $I \leq \GNn$ is called a \emph{monomial stable subgroup} if it is generated by a (possibly infinite) set of monomials.
\end{definition}

For $g \in G$, let $\delta_{g}$ denote the monomial $\N^n \rightarrow G$ that sends $(0, \ldots, 0)$ to $g$, and all other vectors to $e$.
Then for $\alpha \in \N^n$, the monomial $\li{\alpha}(\delta_g)$ sends $\alpha$ to $g$ and all other vectors to $e$.

Let $I$ be a monomial stable subgroup. 
For every $g \in G$, let $A_g$ denote the set of vectors $\alpha \in \N^n$ such that $\li{\alpha}(\delta_g) \in I$:
\[
A_g \coloneqq \{\alpha \in \N^n \mid \li{\alpha}(\delta_{g}) \in I\}.
\]
Then $A_g + \N^n \subseteq A_g$.
The family of sets $A_g, g \in G$, gives us a characterization of $I$.
This yields a version of Dickson's lemma for stable subgroups:

\begin{lemma}[Dickson's lemma for stable subgroups]\label{lem:Dickson_tile}
    Every monomial stable subgroup of $\GNn$ is generated by a finite number of monomials.
\end{lemma}
\begin{proof}
    Let $I \leq \GNn$ be a monomial stable subgroup.
    By Dickson's lemma (Lemma~\ref{lem:Dickson}), for every $g \in G$, the set $A_g$ can be written as $A_g = \{\alpha_{g,1}, \ldots, \alpha_{g,k}\} + \N^n$.
    We claim that the set of monomials 
    \[
    \textstyle
    B \coloneqq \bigcup_{g \in G} \big\{\li{\alpha_{g,1}}(\delta_g), \ldots, \li{\alpha_{g,k}}(\delta_g) \big\}
    \]
    generates $I$.
    Indeed, each monomial $\li{\alpha_{g,j}}(\delta_g)$ is in $I$, so $\gen{B} \subseteq I$.
    To show $I \subseteq \gen{B}$, take any $f \in I$.
    Since $I$ is generated by monomials, $f$ can be written as a product $t_1 t_2 \cdots t_s$, where each $t_i \in I$ is a monomial.
    Write $t_i = \li{\multideg(t_i)}(\delta_g)$ where $g = \LC(t_i)$.
    Then $\multideg(t_i) \in A_g = \{\alpha_{g,1}, \ldots, \alpha_{g,k}\} + \N^n$.
    Therefore $t_i \in \gen{\li{\alpha_{g,1}}(\delta_g), \ldots, \li{\alpha_{g,k}}(\delta_g)} \subseteq \gen{B}$.
    We conclude that $f = t_1 t_2 \cdots t_s \in \gen{B}$.
\end{proof}
%Then $A_g + \N^n \subseteq A_g$.
%Furthermore, for $g, h\in G$, if $\alpha \in A_g$ and $\alpha \in A_h$, then $\alpha \in A_{gh}$ because $\li{\alpha}(\delta_{gh}) = \li{\alpha}(\delta_{g}\delta_{h}) = \li{\alpha}(\delta_{g}) \cdot \li{\alpha}(\delta_{h}) \in I$.
%For each $\alpha \in \N^n$, the set $\varphi(\alpha) \coloneqq \{g \in G \mid \li{\alpha}(\delta_g) \in I\}$ is a subgroup of $G$.
%Indeed, if $g,h \in \varphi(\alpha)$, then $\li{\alpha}(\delta_{gh}) = \li{\alpha}(\delta_{g}\delta_{h}) = \li{\alpha}(\delta_{g}) \cdot \li{\alpha}(\delta_{h}) \in I$, so $gh \in \varphi(\alpha)$.
%Denote by $\mS(G)$ the set of all subgroups of $G$.
%For each monomial stable subgroup $I$, we obtain a map
%\[
%    \varphi \colon \N^n \rightarrow \mS(G), \quad \alpha \mapsto \varphi(\alpha) = \{g \in G \mid \li{\alpha}(\delta_g)\}.
%\]
%Furthermore, for all $\alpha, \beta \in \N^n$, we have $\varphi(\alpha) \leq \varphi(\alpha + \beta)$.
Similar to the classical setting, we can devise a reduction algorithm:

\begin{lemma}[Lead-reduction for stable subgroups]\label{lem:reduction}
Let $<$ be a monomial order on $\N^n$, and $F$ be a set of non-trivial elements in $\GNn$.
Then every $f \in \GNn$ can be written (by an algorithm) as
\begin{equation*}
f = \li{\alpha_1}f_1 \cdots \li{\alpha_m}f_m \cdot r,
\end{equation*}
where $m \in \N$, $f_i \in F^{\pm}, \alpha_i \in \N^n$, and
\begin{enumerate}[nosep, label=(\roman*)]
    \item $\multideg(\li{\alpha_i}f_i) \leq \multideg(f)$.
    \item $r$ is either trivial or $\LT(r) \notin \gen{\LT(h) \colon h \in F}$.
\end{enumerate}
\end{lemma}
\begin{proof}
    The algorithm starts with $r \coloneqq f$ and repeats the following.
    If $r$ is non-trivial and $\LT(r) \in \gen{\LT(h) \colon h \in F}$, then we can write $\LT(r) = t_1 t_2 \cdots t_s$, where each $t_i$ is a monomial of the form $\li{\alpha_i} \LT(f_i),\; f_i \in F^{\pm}$.
    Note that removing all the monomials $t_i$ whose support is different from $\LT(r)$ does not change the product.
    Hence we can suppose $\supp(t_i) = \supp(\LT(r))$ for all $i$. 
    Then 
    \[
    r = \li{\alpha_1} f_1 \li{\alpha_2} f_2 \cdots \li{\alpha_s} f_s \cdot r',
    \]
    for some $r' \in \GNn$, where $\multideg(\li{\alpha_i} f_i) = \multideg(r)$ for all $i$, and $\LT(r) = \LT(\li{\alpha_1} f_1 \li{\alpha_2} f_2 \cdots \li{\alpha_s} f_s)$.
    Therefore $\multideg(r') < \multideg(r)$.
    %Furthermore $\multideg(\li{\alpha_i} f_i) = \multideg(r)$ for all $i$.
    We replace $r$ by $r'$ and repeat the procedure.
    Since $\multideg(r)$ decreases after each repetition and $<$ is a well-order, the algorithm terminates after finitely many steps with either trivial $r$ or $\LT(r) \notin \gen{\LT(h) \colon h \in F}$.
\end{proof}

In Lemma~\ref{lem:reduction}, the tile $r$ is called a \emph{reduction of $f$ by $F$}. 
As in the classical case, $r$ is not unique in general.
For a stable subgroup $I$, define the set
$
\LT(I) \coloneqq \{\LT(f) \mid f \in I\},
$
and let $\gen{\LT(I)}$ be the stable subgroup it generates.
We obtain a Hilbert basis theorem for stable subgroups of $\GNn$:

\begin{corollary}\label{cor:Hilbert_stable}
    Every stable subgroup $I \leq \GNn$ is finitely generated.
\end{corollary}
\begin{proof}
    Fix any monomial order.
    By Lemma~\ref{lem:Dickson_tile}, the monomial stable subgroup $\gen{\LT(I)}$ is generated by a finite number of monomials $t_1, \ldots, t_k \in \LT(I)$.
    For each $t_i$, take $f_i \in I$ such that $\LT(f_i) = t_i$.
    Then $\gen{\LT(I)} = \gen{\LT(f_1), \ldots, \LT(f_k)}$.
    For any $f \in I$, let $r \in I$ be its reduction by $\{f_1, \ldots, f_k\}$.
    Then $\LT(r) \in \LT(I)$ $\subseteq \gen{\LT(f_1), \ldots, \LT(f_k)}$, so $r$ is trivial by definition of reduction, hence $f \in \gen{f_1, \ldots, f_k}$. We conclude that $I = \gen{f_1, \ldots, f_k}$.
\end{proof}

The above proof leads to the definition of a standard basis for stable subgroups of $\GNn$:

\begin{definition}[Standard basis for stable subgroups]\label{def:standard}
    Fix a monomial order on $\N^n$.
    A finite subset $B = \{b_1, \ldots, b_k\}$ of a stable subgroup $I \leq \GNn$ is called a \emph{standard basis}, if
    \begin{equation*}
    \gen{\LT(I)} = \gen{\LT(b_1), \ldots, \LT(b_k)}.
    \end{equation*}
\end{definition}

\begin{theorem}[Membership problem in stable subgroups]\label{thm:membership}
    Let $B = \{b_1, \ldots, b_k\}$ be a standard basis of a stable subgroup $I \leq \GNn$ and $f \in \GNn$. 
    If $f \in I$ then all reductions of $f$ by $B$ are trivial.
    Conversely, if any reduction of $f$ by $B$ is trivial then $f \in I$.
\end{theorem}
\begin{proof}
    Let $f \in I$ and let $r$ be a reduction of $f$ by $B$.
    Then $r \in I$, so $\LT(r) \in \gen{\LT(b_1), \ldots, \LT(b_k)}$ by the definition of standard bases.
    Hence $r$ must be trivial by the definition of reduction.
    Conversely, if there is a trivial reduction of $f$ by $B$ then $f = \li{\alpha_1}f_1 \cdots \li{\alpha_m}f_m$ with $f_i \in B^{\pm}$, so $f \in I$.
\end{proof}

As in the classical setting, Definition~\ref{def:standard} does not require $b_1, \ldots, b_k$ to generate $I$, but it follows from Theorem~\ref{thm:membership} that a standard basis of $I$ generates $I$.

\subsection{Computing a standard basis}\label{subsec:basis_comp}
So far everything is similar to the classical case. This changes when we start to \emph{compute} a standard basis.
In the classical setting, an algorithm is constructed with Buchberger's criterion as the termination condition (Theorem~\ref{thm:criterion}).
Recall the key argument: if the leading coefficients of a sum of polynomials cancel out, then the subsum with the highest multidegree is a linear combination of S-polynomials.
However, this argument fails for non-abelian $G$.
The following examples illustrate some main obstacles:

\begin{enumerate}[nosep, label=(\arabic*)]
    \item Let $f = \li{\alpha_1}b_1 \li{\alpha_2}b_2 \li{\alpha_3}b_3$, with $\multideg(\li{\alpha_2}b_2) < \multideg(\li{\alpha_1}b_1) = \multideg(\li{\alpha_3}b_3)$.
    The leading factors in the product are $\li{\alpha_1}b_1$ and $\li{\alpha_3}b_3$.
    Naively, this means we should define $\li{\alpha_1}b_1 \li{\alpha_3}b_3$ as a (translation of) ``S-polynomial''.
    However, since $\li{\alpha_1}b_1 \li{\alpha_2}b_2 \li{\alpha_3}b_3 \neq (\li{\alpha_1}b_1 \li{\alpha_3}b_3) \li{\alpha_2}b_2$, we cannot group up the leading factors to recreate an ``S-polynomial'' as in the classical setting.
    
    \item To make matters worse, the cancellation of leading terms may be for trivial reasons and not due to ``S-polynomials''.
    For example, let $f = \li{\alpha_1}(b_1^{-1}) \li{\alpha_2}b_2 \li{\alpha_1}b_1$ with $\multideg(\li{\alpha_2}b_2) < \multideg(\li{\alpha_1}b_1)$, then the leading factors in the product cancel out for the trivial reason $\li{\alpha_1}(b_1^{-1}) \li{\alpha_1}b_1 = e^{\Z^n}$, and we cannot simplify it by applying reduction.
\end{enumerate}
\smallskip

To overcome these obstacles, notice the following.
In example~(1), we can write $\li{\alpha_1}b_1 \li{\alpha_2}b_2 \li{\alpha_3}b_3 = (\li{\alpha_1}b_1 \li{\alpha_3}b_3) \cdot \big(\li{\alpha_2}b_2)^{(\li{\alpha_3}b_3)}$, where $\big(\li{\alpha_2}b_2)^{(\li{\alpha_3}b_3)} = (\li{\alpha_3}b_3)^{-1} \li{\alpha_2}b_2 \li{\alpha_3}b_3$ is the (right-)conjugate of $\li{\alpha_2}b_2$ by $\li{\alpha_3}b_3$.
This allows us to proceed with reducing the product $\li{\alpha_1}b_1 \li{\alpha_3}b_3$.
In example~(2), $f = \li{\alpha_1}(b_1^{-1}) \li{\alpha_2}b_2 \li{\alpha_1}b_1$ is already a conjugate $\big(\li{\alpha_2}b_2)^{(\li{\alpha_1}b_1)}$.
This shows that in the stable subgroup setting, the ``correct'' way to define leading factors of a product needs to include \emph{conjugates} of the factors.
Hence, in order to generalize S-polynomials to the stable subgroup setting, we need to take into account both conjugation and multiplication.
This motivates our following definition of \emph{S-tiles}.

For tiles $f,g \in \GNn$, let $g^f$ denote the (right) conjugate of $g$ by $f$:
\[
g^f \coloneqq f^{-1}gf.
\]
It is easy to see that $\supp(g^f) = \supp(g)$.
For integers $x \leq y$, define $[x, y] \coloneqq \{x, x+1, \ldots, y\}$.
For $f \in \GNn$, define
\[
\|f\| \coloneqq \min\{d \in \N \mid \supp(f) \subseteq [0, d]^n\}.
\]
That is, $\|f\| \in \N$ is the side length of the smallest cube containing $0$ and $\supp(f)$.
For a set of tiles $B \subset \GNn$, we define an \emph{S-tile} associated to $B$ to be a tile of the form
\begin{equation}\label{eq:S_tile}
    (\li{\alpha_1}b_1)^{f_1} \cdot (\li{\alpha_2}b_2)^{f_2} \cdot \cdots \cdot (\li{\alpha_m}b_m)^{f_m},
\end{equation}
where $m \in \N$, and
\begin{enumerate}[nosep, label=(\roman*)]
    \item $b_1, \ldots, b_m \in B^{\pm}$,
    \item $\alpha_1, \ldots, \alpha_m \in [0,\, 2\max_{b \in B} \|b\|]^n$,
    \item $f_1, \ldots, f_m \in \gen{B}$.
\end{enumerate}
\smallskip

Intuitively, S-tiles describe how conjugates of tiles in $B^{\pm}$ can interact with each other within a radius of $2\max_{b \in B} \|b\|$.
See Figures~\ref{fig:Stilef}-\ref{fig:Stilefgh} for illustration.
Note that we do not \emph{a priori} give a bound on the product length $m$, nor require the leading terms of $(\li{\alpha_i}b_i)^{f_i}$ to align or cancel.

\begin{figure}[h!]
    %for 10 page format move below enumerate and remove smallskip
    \centering
    \begin{minipage}[t]{.20\textwidth}
        \centering
        \includegraphics[width=\textwidth,height=0.8\textheight,keepaspectratio, trim={6.3cm 1.1cm 5.84cm 1.05cm},clip]{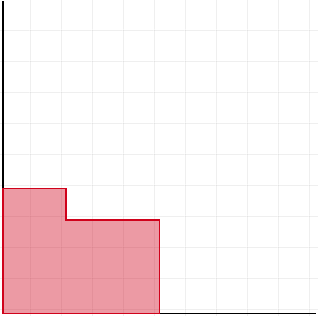}
        \caption{tile $b_1$}
        \label{fig:Stilef}
    \end{minipage}
    \hfill
    \begin{minipage}[t]{.20\textwidth}
        \centering
        \includegraphics[width=\textwidth,height=0.8\textheight,keepaspectratio, trim={6.3cm 1.1cm 5.84cm 1.05cm},clip]{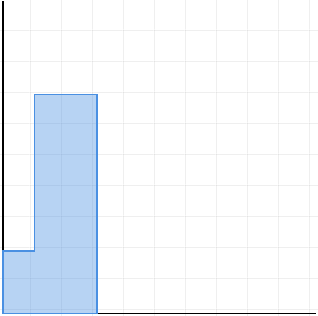}
        \caption{tile $b_2$}
        \label{fig:Stileg}
    \end{minipage}
    \hfill
    \begin{minipage}[t]{.20\textwidth}
        \centering
        \includegraphics[width=\textwidth,height=0.8\textheight,keepaspectratio, trim={6.3cm 1.1cm 5.84cm 1.05cm},clip]{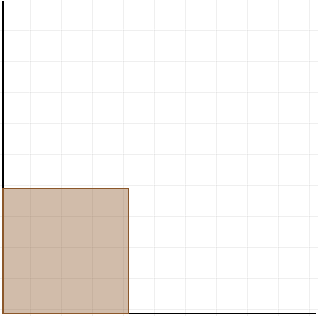}
        \caption{tile $b_3$}
        \label{fig:Stileh}
    \end{minipage}
    \hfill
    \begin{minipage}[t]{0.30\textwidth}
        \centering
        \includegraphics[width=0.67\textwidth,height=0.8\textheight,keepaspectratio, trim={6.3cm 1.1cm 5.84cm 1.05cm},clip]{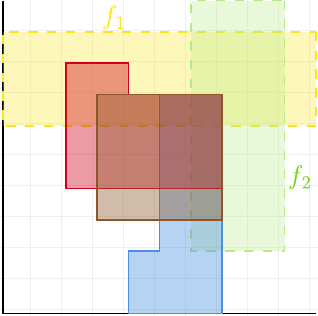}
        \captionsetup{justification=centering}
        \caption{an S-tile \\ $(\li{(2, 4)}b_1)^{f_1} \cdot (\li{(4, 0)}b_2)^{f_2} \cdot (\li{(3, 3)}b_3)$.}
        \label{fig:Stilefgh}
    \end{minipage}
\end{figure}

By definition, S-tiles are the elements in the (non-stable) group generated by $(\li{\alpha}b)^{f} \in \GNn$ where $b \in B, \alpha \in [0,\, 2 \max_{b \in B} \|b\|]^n$, and $f \in \gen{B}$.
When $B$ is finite, the number of S-tiles associated to $B$ is finite:
this is because 
\[
\supp\left((\li{\alpha}b)^{f}\right) = \supp(\li{\alpha}b) = \alpha + \supp(b) \subseteq \alpha + [0, \max_{b \in B} \|b\|]^n \subseteq [0,\, 3 \max_{b \in B} \|b\|]^n,
\]
so all S-tiles associated to $B$ are elements of the finite group $G^{[0,\, 3 \max_{b \in B} \|b\|]^n} \leq \GNn$.

%\noindent A crucial observation is that the number of distinct S-tiles is finite if the set $B$ is finite.
%Indeed, each $b_i \in B^{\pm}$ has only finitely many distinct conjugations since $\supp(b_i)$ is finite.
%Hence, there are only finitely many distinct tiles of the form $(\li{\alpha_i}b_i)^{f_i}, b_i \in B^{\pm}, \|\alpha_i\| \leq 2\max_{b \in B} \|b\|$, so the subgroup (not shift-stable subgroup) they generate is therefore also finite.
%In Theorem~\ref{thm:compute_tile_basis} below, 
%It is not difficult to show that the set of S-tiles can be computed from the set $B$ (see Theorem~\ref{thm:compute_tile_basis} below).
We now prove a ``Buchberger's criterion'' in the stable subgroup setting, where we use S-tiles to provide a test for whether a set of generators is a standard basis:

\begin{theorem}\label{thm:criterion_tile}
    Let $I \leq \GNn$ be a stable subgroup and $B$ be a finite set of non-trivial tiles generating $I$.
    Fix a monomial order on $\N^n$.
    Then $B$ is a standard basis if and only if every S-tile associated to $B$ has a trivial reduction by $B$.
\end{theorem}
\begin{proof}
    If $B$ is a standard basis, then by Theorem~\ref{thm:membership} every S-tile associated to $B$ reduces to the trivial tile by $B$.
    
    For the other implication, suppose that every S-tile associated to $B$ has a trivial reduction by $B$.
    We will show that every $f \in I$ satisfies $\LT(f) \in \gen{\LT(b) \colon b \in B}$.
    
    Take $f \in I = \gen{B}$.
    Consider a representation of $f$ as a product
    \begin{equation}\label{eq:f_product_of_conjugates}
    f = (\li{\beta_1}b_1)^{f_1} \cdot (\li{\beta_2}b_2)^{f_2} \cdot \cdots \cdot (\li{\beta_k}b_k)^{f_k},
    \end{equation}
    where $b_i \in B^{\pm}$, $\beta_i \in \N^n$, and $f_i \in \gen{B}$.
    Such a representation always exists, for example by taking all $f_i$ to be the trivial tile.
    Define
    \[
    \delta \coloneqq \max_{1 \leq \ell \leq m} \multideg(\li{\beta_\ell}b_\ell) = \max_{1 \leq \ell \leq m} \multideg(\li{\beta_\ell}b_\ell)^{f_\ell}.
    \]
    Among all possible representations of $f$ in the form~\eqref{eq:f_product_of_conjugates}, pick one where $\delta$ is minimal with respect to the monomial order $<$.
    The minimal $\delta$ exists because $<$ is a well-order.
    Let $u_1, \ldots, u_s$ be all the factors $(\li{\beta_\ell}b_\ell)^{f_{\ell}}$ in Equation~\eqref{eq:f_product_of_conjugates} where $\multideg(\li{\beta_\ell}b_\ell)$ attains $\delta$.
    We can write 
    \[
    f = t_{01} \cdots t_{0j_0} u_1 t_{11} \cdots t_{1j_1} u_2 \cdots u_s t_{s1} \cdots t_{sj_s},
    \]
    where $t_{ij}, u_i$ are tiles of the form $(\li{\beta_{\ell}}b_{\ell})^{f_{\ell}}$, and
    \[
    \multideg(u_i) = \delta > \multideg(t_{ij}).
    \]
    Moving all the $u_i$'s to the left using conjugation, we have 
    \begin{equation}\label{eq:f_collect_LT}
    f = u_1 u_2 \cdots u_s \cdot \underbrace{\left(t_{01}^{u_1 u_2 \cdots u_s} \cdots t_{0 i_0}^{u_1 u_2 \cdots u_s} \cdot t_{11}^{u_2 \cdots u_s} \cdots t_{ i_1}^{u_2 \cdots u_s} \cdots t_{s1} \cdots t_{sj_s} \right)}_{\text{each term has } \multideg < \delta}.
    \end{equation}
    %We now write 
    %$
    %u_i = (\li{\beta_i}b_i)^{f_i}
    %$
    %where $b_i \in B^{\pm}, \beta_i \in \N, f_i \in \gen{B}$, these are tiles with multidegree $\delta$.

    %This is not evident because in $u_i = (\li{\beta_i}b_i)^{f_i}$, the conjugation by $f_i$ is applied \emph{after} the translation by $\beta_i$.
    We now proceed in two steps.
    In the first step we will express the leading product $u_1 u_2 \cdots u_s$ as a translation of an S-tile.
    In the second step we proceed in two cases according to whether $\multideg(f)=\delta$ or $\multideg(f) < \delta$, similar to the proof of Theorem~\ref{thm:criterion} in the classical setting.
    \smallskip

    \textbf{Step 1.} 
    We express $u_1 u_2 \cdots u_s$ as the translation of an S-tile.
    For each $i$, write 
    $
    u_i = (\li{\beta_i}b_i)^{f_i}
    $
    where $b_i \in B^{\pm}, \beta_i \in \N^n, f_i \in \gen{B}$.
    We will find a vector $\gamma \in \N^n$ satisfying 
    \[
    \beta_i - \gamma \in [0,\, 2 \max_{b \in B} \|b\|]^n,\quad i = 1, \ldots, s,
    \]
    as well as tiles $f'_1, \ldots, f'_s \in \gen{B}$, such that 
    \[
    u_i = (\li{\beta_i}b_i)^{f_i} = \li{\gamma} {\Big((\li{\beta_i -\gamma}b_i)^{f'_i} \Big)}, \quad i = 1, \ldots, s.
    \]
    Then each $(\li{\beta_i -\gamma}b_i)^{f'_i}$ will be an S-tile associated to $B$.
    See Figures~\ref{fig:conjugate}, \ref{fig:conjugateT}, \ref{fig:conjugateS} for an illustration.
    %A naive attempt to write $(\li{\beta_i}b_i)^{f_i}$ as a $\gamma$-translation of an S-tile would be to take some $\gamma \in \N^n$ not too far from $\delta$, and write
    %\[
    %(\li{\beta_i}b_i)^{f_i} = f_i^{-1} \cdot \li{\beta_i}b_i \cdot f_i = \li{\gamma} {\big(\li{-\gamma}f_i^{-1} \cdot \li{\beta_i - \gamma}b_i \cdot \li{-\gamma}f_i \big)} = \li{\gamma} {\Big((\li{\beta_i - \gamma}b_i)^{(\li{- \gamma}f_i)} \Big)}.
    %\]
    %Note that $(\beta_i - \gamma) + \multideg(b_i) = \multideg(\li{\beta_i}b_i) - \gamma = \delta - \gamma$, so if we take $\gamma$ such that $\|\delta - \gamma\| \leq 2 \max_{b \in B} \|b\|$, then condition~(iii) in the definition~\eqref{eq:S_tile} of S-tiles will be satisfied.
    %However, the problem with this attempt is that $\li{- \gamma}f_i$ is not well-defined because the vector $ - \gamma$ has negative entries.
    %Nevertheless, this can be remedied by finding an equivalent conjugator $f'_i$:

\begin{figure}[b!]
    \centering
    \begin{minipage}[t]{.27\textwidth}
        \centering
        \includegraphics[width=\textwidth,height=0.8\textheight,keepaspectratio, trim={6.3cm 1.1cm 5.84cm 1.05cm},clip]{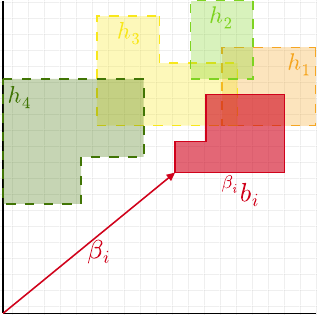}
        \captionsetup{justification=centering}
        \caption{the tile $u_i = (\li{\beta_i}b_i)^{f_i}$, where $f_i = h_1 h_2 h_3 h_4$.}
        \label{fig:conjugate}
    \end{minipage}
    \hfill
    \begin{minipage}[t]{.27\textwidth}
        \centering
        \includegraphics[width=\textwidth,height=0.8\textheight,keepaspectratio, trim={6.3cm 1.1cm 5.84cm 1.05cm},clip]{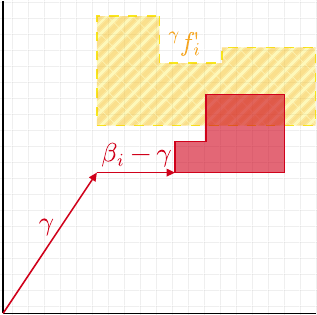}
        \captionsetup{justification=centering}
        \caption{the tile $\li{\gamma} {\Big((\li{\beta_i -\gamma}b_i)^{f'_i} \Big)}$, equal to $u_i$}
        \label{fig:conjugateT}
    \end{minipage}
    \hfill
    \begin{minipage}[t]{0.27\textwidth}
        \centering
        \includegraphics[width=\textwidth,height=0.8\textheight,keepaspectratio, trim={6.3cm 1.1cm 5.84cm 1.05cm},clip]{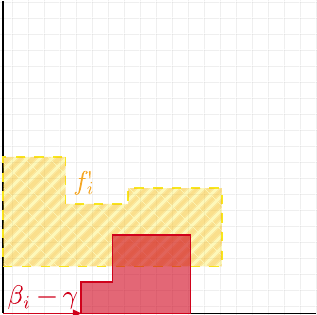}
        \captionsetup{justification=centering}
        \caption{the S-tile $(\li{\beta_i -\gamma}b_i)^{f'_i}$.}
        \label{fig:conjugateS}
    \end{minipage}
\end{figure}

    For an index $i \in \{1, \ldots, s\}$, write 
    \begin{equation*}
    f_i = \li{\alpha_{i1}} b_{i1} \cdot \li{\alpha_{i2}} b_{i2} \cdot \cdots
    \end{equation*}
    where $b_{ij} \in B^{\pm}, \alpha_{ij} \in \N^n$.
    So
    \[
    u_i = (\li{\beta_i}b_i)^{f_i} = (\li{\beta_i}b_i)^{\li{\alpha_{i1}} b_{i1} \cdot \li{\alpha_{i2}} b_{i2} \cdot \cdots}.
    \]
    %The key observation is that, without changing $(\li{\beta_i}b_i)^{f_i}$, we can suppose without loss of generality that $\supp(\li{\alpha_{ij}} b_{ij}) \cap \supp(\li{\beta_i}b_i)) \neq \emptyset$ for all $j$.
    Since any factor $\li{\alpha_{ij}} b_{ij}$ whose support is disjoint from
    $\supp((\li{\beta_i}b_i)^{\li{\alpha_{i1}} b_{i1}  \cdots \li{\alpha_{i(j-1)}} b_{i(j-1)}}) = \supp(\li{\beta_i}b_i)$ acts trivially on it by conjugation, we may delete all such factors from $f_i = \li{\alpha_{i1}} b_{i1} \cdot \li{\alpha_{i2}} b_{i2} \cdot \cdots$, and assume without loss of generality that 
    \[
    \supp(\li{\alpha_{ij}} b_{ij}) \cap \supp(\li{\beta_i}b_i) \neq \emptyset, \quad j = 1, 2, \ldots.
    \]
    Take any $\lambda_{ij} \in \supp(\li{\alpha_{ij}} b_{ij}) \cap \supp(\li{\beta_i}b_i))$, we obtain
    \begin{align}\label{eq:da}
    \delta - \alpha_{ij} & = (\delta - \lambda_{ij}) + (\lambda_{ij} - \alpha_{ij}) \nonumber\\
    & \in [- \|b_i\|,\, \|b_i\| ]^n + (\lambda_{ij} - \alpha_{ij}) && \text{(since $\delta, \lambda_{ij} \in \supp(\li{\beta_i}b_i)$)} \nonumber\\
    & \subseteq [- \|b_i\|,\, \|b_i\| ]^n + [0,\, \|b_{ij}\|]^n  && \text{(since $\lambda_{ij} \in \supp(\li{\alpha_{ij}} b_{ij})$)} \nonumber\\
    & \subseteq [- \max_{b \in B} \|b\|,\, 2 \max_{b \in B} \|b\| ]^n.
    \end{align}
    Also,
    \begin{equation}\label{eq:db}
    \delta - \beta_i = \multideg(\li{\beta_i}b_i) - \beta_i = \multideg(b_i) \in [0,\, 2 \max_{b \in B} \|b\| ]^n.
    \end{equation}
    We then take $\gamma$ to be the coordinate-wise minimum of all $\alpha_{ij}$ and $\beta_i$, then $\alpha_{ij} -\gamma, \beta_i - \gamma \in \N^n$ for all $i, j$.
    Define
    \[
    f'_i \coloneqq \li{(\alpha_{i1} -\gamma)} b_{i1} \cdot \li{(\alpha_{i2} -\gamma)} b_{i2} \cdot \cdots.
    \]
    Then $f_i = \li{\gamma}f'_i$, and
    \begin{equation}\label{eq:u_s}
    u_i = (\li{\beta_i}b_i)^{f_i} = f_i^{-1} \cdot \li{\beta_i}b_i \cdot f_i = \li{\gamma} {\big({f'_i}^{-1} \cdot \li{\beta_i - \gamma}b_i \cdot f'_i \big)} = \li{\gamma} {\Big((\li{\beta_i -\gamma}b_i)^{f'_i} \Big)}.
    \end{equation}
    Here, $\beta_i - \gamma \in \N^n$ is coordinate-wise smaller than $\delta - \gamma$ (by~\eqref{eq:db}).
    Since $\delta - \gamma$ is the coordinate-wise maximum of all $\delta - \alpha_{ij}$ and $\delta - \beta_i$, each coordinate of $\delta - \gamma$ is smaller than $2 \max_{b \in B} \|b\|$ (by~\eqref{eq:da} and \eqref{eq:db}).
    Therefore
    \[
    \beta_i - \gamma \in [0,\, 2 \max_{b \in B} \|b\| ]^n, \quad i = 1, \ldots, s.
    \]
    So each $(\li{\beta_i -\gamma}b_i)^{f'_i}$ is an S-tile associated to $B$.

    Setting $\alpha_i \coloneqq \beta_i - \gamma$, Equation~\eqref{eq:u_s} allows us to write the term $u_1 u_2 \cdots u_s$ as a translation of an S-tile:
    \begin{align}\label{eq:f_product_S_tile}
        f & = u_1 u_2 \cdots u_s \cdot \left(t_{01}^{u_1 u_2 \cdots u_s} \cdots t_{0 i_0}^{u_1 u_2 \cdots u_s} \cdot t_{11}^{u_2 \cdots u_s} \cdots t_{ i_1}^{u_2 \cdots u_s} \cdots t_{s1} \cdots t_{sj_s} \right) \nonumber\\
        & = \li{\gamma} {\Big((\li{\beta_1 -\gamma}b_1)^{f'_1} \Big)} \cdot \li{\gamma} {\Big((\li{\beta_2 -\gamma}b_2)^{f'_2} \Big)} \cdots \li{\gamma} {\Big((\li{\beta_s -\gamma}b_s)^{f'_s} \Big)} \cdot \underbrace{\big(\cdots \cdots\big)}_{\text{each} \multideg < \delta} \nonumber\\
        & = \li{\gamma} {\Big((\li{\alpha_1}b_1)^{f'_1} \cdot (\li{\alpha_2}b_2)^{f'_2} \cdot \cdots \cdot  (\li{\alpha_s}b_s)^{f'_s} \Big)} \cdot \underbrace{\big(\cdots \cdots\big)}_{\text{each} \multideg < \delta},
    \end{align}
    %Here, each $\alpha_i \in \N^n$ satisfies 
    %\[
    %\|\alpha_i\| \leq \|\alpha_i + \multideg(b_i)\| = \|\beta_i - \gamma + \multideg(b_i)\| = \|-\gamma + \multideg(\li{\beta_i}b_i)\| = \|\delta - \gamma\| \leq 2\max_{b \in B} \|b\|.
    %\]
    where $(\li{\alpha_1}b_1)^{f'_1} \cdot (\li{\alpha_2}b_2)^{f'_2} \cdot \cdots \cdot (\li{\alpha_s}b_s)^{f'_s}$ is an S-tile associated to $B$.
    \smallskip

    \textbf{Step 2.}
    Since every S-tile admits a trivial reduction by $B$, we can write
    \begin{equation}\label{eq:reduction_S}
    (\li{\alpha_1}b_1)^{f'_1} \cdot (\li{\alpha_2}b_2)^{f'_2} \cdot \cdots \cdot (\li{\alpha_s}b_s)^{f'_s} = \li{\alpha'_1}b'_1 \cdot \cdots \cdot \li{\alpha'_k}b'_k
    \end{equation}
    for some $b'_i \in B^{\pm}, \alpha'_i \in \N^n$, where 
    \[
    \multideg(\li{\alpha'_i}b'_i) \leq \multideg\left((\li{\alpha_1}b_1)^{f'_1} \cdot \cdots \cdot (\li{\alpha_s}b_s)^{f'_s}\right), \quad i = 1, \ldots, k.
    \]
    Substitute the first product in~\eqref{eq:f_product_S_tile} using~\eqref{eq:reduction_S}, we obtain
    \begin{equation}\label{eq:smaller_f}
    f = \li{(\alpha'_1 + \gamma)}b'_1 \cdot \cdots \cdot \li{(\alpha'_k + \gamma)}b'_k \cdot \underbrace{\big(\cdots \cdots\big)}_{\text{each} \multideg < \delta},
    \end{equation}
    where
    \begin{equation}\label{eq:md_leq}
    \multideg \left(\li{(\alpha'_i + \gamma)}b'_i\right) \leq \multideg\li{\gamma}{\left((\li{\alpha_1}b_1)^{f'_1} \cdot \cdots \cdot (\li{\alpha_s}b_s)^{f'_s}\right)}, \quad i = 1, \ldots, k.
    \end{equation}
    Note that $\multideg(f) \leq \delta$.
    As in the classical setting, there are two cases.
    
    \textbf{Case 1:} $\multideg(f) = \delta$. 
    Then
    \[
    \LT(f) = \LT \left(\li{(\alpha'_1 + \gamma)}b'_1 \cdot \cdots \cdot \li{(\alpha'_k + \gamma)}b'_k \right) \in \gen{\LT(b) \colon b \in B},
    \]
    which is what we want to prove.
    
    \textbf{Case 2:} $\multideg(f) < \delta$.
    Then by Equations~\eqref{eq:md_leq} and~\eqref{eq:f_product_S_tile}, we have
    \[
    \multideg \left(\li{(\alpha'_i + \gamma)}b'_i\right) \leq \multideg\li{\gamma} {\Big((\li{\alpha_1}b_1)^{f'_1} \cdot (\li{\alpha_2}b_2)^{f'_2} \cdot \cdots \cdot  (\li{\alpha_s}b_s)^{f'_s} \Big)} = \multideg \Big( f \cdot \underbrace{\big(\cdots \cdots\big)^{-1}}_{\text{each} \multideg < \delta} \Big) < \delta
    \]
    for $i = 1, \ldots, k$.
    Therefore, in the representation~\eqref{eq:smaller_f} of $f$, each factor has multidegree strictly smaller than $\delta$.
    Compared to the representation~\eqref{eq:f_product_of_conjugates} where some terms have multidegree $\delta$, this contradicts the minimality of $\delta$ and completes the proof of the theorem.
\end{proof}

Using Theorem~\ref{thm:criterion_tile}, we obtain a ``Buchberger's algorithm'' for computing a standard basis:

\begin{theorem}\label{thm:compute_tile_basis}
    Fix a monomial order.
    There is an algorithm which, given as input a finite set of generators for a stable subgroup $I$, computes a standard basis for $I$.
\end{theorem}
\begin{proof}
    Let $B$ be the given set of generating tiles for $I$.
    Similar to the ideal setting, the algorithm starts with the set $B$ and repeatedly adds to $B$ the reductions of S-tiles.
    Formally, the algorithm is given by Algorithm~\ref{alg:compute_tile_basis}.

    \begin{algorithm}[h]
    \caption{Computing a standard basis for $I = \gen{B}$}\label{alg:compute_tile_basis}
    \Repeat{$B=B'$}{
        $B' \coloneqq B$\;
        Compute the set $S$ of all S-tiles associated to $B$\;
        \ForEach{$s\in S$}{
            Let $r$ be a reduction of $s$ by $B$\;
            \lIf{$r$ is not trivial}{$B \coloneqq B\cup\{r\}$}
        }
    }
    \Return{$B$}\;
    \end{algorithm}

    If Algorithm~\ref{alg:compute_tile_basis} terminates, then it outputs a set $B$ satisfying the criterion in Theorem~\ref{thm:criterion_tile}, and is therefore a standard basis.
    To obtain the proof of Theorem~\ref{thm:compute_tile_basis}, it suffices to show that (i) the set of S-tiles associated to $B$ is computable, and that (ii) Algorithm~\ref{alg:compute_tile_basis} terminates.

    \smallskip
    \begin{adjustwidth}{3mm}{3mm}
    \textbf{Claim.} Let $B$ be a finite set of tiles. Then the set of S-tiles associated to $B$ is finite and computable.
    \smallskip
    
    \noindent$\blacktriangleright$
    Recall that an S-tile is of the form $(\li{\alpha_1}b_1)^{f_1} \cdot (\li{\alpha_2}b_2)^{f_2} \cdot \cdots \cdot (\li{\alpha_k}b_m)^{f_m}$, where $b_i \in B^{\pm}, \alpha_i \in [0,\, 2\max_{b \in B} \|b\|]^n$, and $f_i \in \gen{B}$.  
    That is, S-tiles are elements of the group generated by
    \[
    T \coloneqq \big\{(\li{\alpha}b)^{f} \;\big|\; b \in B^{\pm}, \alpha \in [0, 2\max_{b \in B} \|b\|]^n, f \in \gen{B} \big\}.
    \]
    First we show that $T$ is finite and computable.
    
    Fix a tile $t = \li{\alpha}b$.
    Its set of conjugates $\{t^{f} \mid f \in \gen{B}\}$ is finite because $\supp(t^{f}) = \supp(t)$ is finite.
    The set $\{t^{f} \mid f \in \gen{B}\}$ is computable by starting with $C = \{t\}$ and repeating the following procedure: for each $c \in C$ and $b \in B^{\pm}$, test whether $c^b \in C$; if $c^b \notin C$ then add it to $C$.
    This procedure terminates since $C$ becomes larger after each step.
    At termination, we have $c^b \in C$ for all $c \in C, b \in B^{\pm}$, so the set $C$ is closed under conjugation by $B^{\pm}$ and we obtain $C = \{t^{f} \mid f \in \gen{B}\}$.
    
    The set $T$ is the finite union of the above sets: 
    \[
    T = \bigcup_{t = \li{\alpha}b, b \in B^{\pm}, \alpha \in [0, 2\max_{b \in B} \|b\|]^n}\{t^{f} \mid f \in \gen{B}\}.
    \]
    Hence $T$ is finite and computable.
    Finally, the set of S-tiles is the group generated by $T$.
    Since all elements of $T$ have support contained in a fixed finite set, this group is finite. It is computable by enumerating products of elements of $T$ until the set of products is closed under multiplication.
    \hfill $\blacktriangleleft$
    \end{adjustwidth}
    \smallskip

    \begin{adjustwidth}{3mm}{3mm}
    \textbf{Claim.} Algorithm~\ref{alg:compute_tile_basis} terminates.
    \smallskip
    
    \noindent$\blacktriangleright$
    We show that the repeat loop in Algorithm~\ref{alg:compute_tile_basis} can only repeat finitely many times.
    Suppose the contrary, then we have an infinite chain of tile sets $B_1 \subset B_2 \subset \cdots$ in the stable subgroup $I$, where $B_{i+1} \setminus B_{i} = \{r_i\}$ and $r_i$ is a reduction of an S-tile by $B_i$.
    By the definition of reduction, $\LT(r_i) \notin \gen{\LT(b) \colon b \in B_i}$.
    Consider the monomial stable subgroup $M = \gen{\LT(b) \colon b \in \bigcup_{i \in \N} B_i}$.
    By Dickson's lemma for stable subgroups (Lemma~\ref{lem:Dickson_tile}), $M$ is generated by finitely many monomials.
    Each of the generators of $M$ is in $\gen{\LT(b) \colon b \in B_i}$ for some $i \in \N$.
    Taking $k$ to be the maximum of these $i$ we obtain $M = \gen{\LT(b) \colon b \in B_k}$.
    A contradiction since $\LT(r_k) \notin \gen{\LT(b) \colon b \in B_k}$.
    \hfill $\blacktriangleleft$
    \end{adjustwidth}
    \smallskip
    Combining the two claims, we conclude that Algorithm~\ref{alg:compute_tile_basis} is well-defined and terminates.
\end{proof}

\subsection{Saturation}\label{subsec:saturation}
Now that we have constructed an algorithm for computing standard bases, we can use it to solve several algorithmic problems for stable subgroups of $\GNn$.
By Theorem~\ref{thm:membership}, a standard basis immediately gives us an algorithm to decide membership in $I \leq \GNn$: a tile $f \in \GNn$ is in $I$ if and only if its reduction by a standard basis of $I$ is trivial.
The next problem we need to tackle is \emph{subgroup saturation}.

In later sections of this paper, we will often need to consider tiles in $G^{(\Z^{m} \times \N^{n-m})}$ or $G^{(\N^{n-m} \times \Z^{m})}$ instead of $\GNn$.
That is, in some directions we need to allow the tiles to translate by a negative integer.
For $0 \leq m \leq n$, the restricted direct product $G^{(\Z^{m} \times \N^{n-m})}$ is defined similarly to $\GNn$: its elements are finitely supported maps 
\[
\Z^{m} \times \N^{n-m} \rightarrow G.
\]
The group $G^{(\Z^{m} \times \N^{n-m})}$ is equipped with a shift action by $\Z^{m} \times \N^{n-m}$, and shift-stable subgroups of $G^{(\Z^{m} \times \N^{n-m})}$ are defined accordingly.
For $S \subseteq G^{(\Z^{m} \times \N^{n-m})}$, we let 
\[
\gen{S}_{\Z^{m} \times \N^{n-m}} \leq G^{(\Z^{m} \times \N^{n-m})}
\]
denote the $(\Z^{m} \times \N^{n-m})$-shift-stable subgroup it generates: when $m = 0$, this is simply $\gen{S}$.
For $S \subseteq G^{(\N^{n-m} \times \Z^{m})}$, the $(\N^{n-m} \times \Z^{m})$-shift-stable subgroup 
\[
\gen{S}_{\N^{n-m} \times \Z^{m}} \leq G^{(\N^{n-m} \times \Z^{m})}
\]
is defined similarly.

In the classical setting of ideals, negative translation corresponds to ideals in \emph{Laurent polynomial rings} such as $\Z[X, X^{-1}]$.
There are essentially three approaches to deal with negative exponents in the ideal setting.
The first approach is to assign a new variable $Y$ for $X^{-1}$ and write the ring $\Z[X, X^{-1}]$ as the quotient $\Z[X, Y]/\gen{XY-1}$, thereby reducing to the case of regular polynomial rings.
The second approach is to start with an ideal $I \subseteq \Z[X]$ and compute the \emph{saturation ideal} 
\[
I \colon X^{\infty} \coloneqq \{f \in \Z[X] \mid \exists t \in \N \text{ such that } X^t f \in I\} \subseteq \Z[X].
\]
Then the ideal of $\Z[X, X^{-1}]$ generated by $I$ is equal to $\{X^{-t}f \mid t \in \N,\; f \in I \colon X^{\infty}\}$.
The third approach is to develop a notion of Gr\"{o}bner basis specifically for Laurent polynomial rings: this is particularly suited to computation in toric varieties~\cite{Sturmfels1996}.

In order to generalize to the stable subgroup setting, the first approach does not work due to $G$ being non-abelian.
Among the remaining two approaches, we will generalize the saturation approach due to its relative simplicity.

Let $I \leq \GNn$ be a stable subgroup, and let $\alpha \in \N^n$ be a vector.
The \emph{saturation subgroup} of $I$ with respect to $\alpha$ is the stable subgroup
\[
I \colon \alpha^{\infty} \coloneqq \left\{f \in G^{(\N^n)} \;\middle|\; \exists t \in \N \text{ such that } \li{t\alpha}f \in I \right\} \leq \GNn.
\]
Let 
\[
\varepsilon_1 = (1, 0, \ldots, 0),\; \ldots,\; \varepsilon_n = (0, \ldots, 0, 1) \in \N^n,
\]
denote the canonical basis of $\N^n$.
We now primarily consider the case $\alpha = \varepsilon_n$.
First we show that computing $I \colon \varepsilon_n^{\infty}$ allows us to represent shift-stable subgroups of $G^{(\N^{n-1} \times \Z)}$:

\begin{lemma}\label{lem:negative_saturation}
    Let $S \subseteq \GNn$ be a set of tiles and $I = \gen{S}$.
    %Let $\gen{S}_{\N^{n-1} \times \Z} \leq G^{(\N^{n-1} \times \Z)}$ denote the $(\N^{n-1} \times \Z)$-shift-stable subgroup generated by $S$.
    Then
    \[
    \gen{S}_{\N^{n-1} \times \Z} = \left\{ \li{-t \varepsilon_n} f \;\middle|\; t \in \N, \; f \in I \colon \varepsilon_n^{\infty}\right\} \leq G^{(\N^{n-1} \times \Z)}.
    \]
\end{lemma}
\begin{proof}
    First we show $\gen{S}_{\N^{n-1} \times \Z} \subseteq  \left\{ \li{-t \varepsilon_n} f \;\middle|\; t \in \N, \; f \in I \colon \varepsilon_n^{\infty}\right\}$.
    Let $h \in \gen{S}_{\N^{n-1} \times \Z}$, then we can write
    \[
    h = \li{\alpha_1} s_1 \li{\alpha_2} s_2 \cdots \li{\alpha_k} s_k
    \]
    with $\alpha_i \in \N^{n-1} \times \Z$ and $s_i \in S^{\pm}$.
    Take a large enough $t \in \N$ such that $\alpha_i + t \varepsilon_n \in \N^n$ for $i = 1, \ldots, k$. 
    Then
    \[
    \li{t \varepsilon_n}h = \li{\alpha_1 + t \varepsilon_n} s_1 \cdot \li{\alpha_2 + t \varepsilon_n} s_2 \cdot \cdots \cdot \li{\alpha_k + t \varepsilon_n} s_k \in \gen{S} = I.
    \]
    Therefore $\li{t \varepsilon_n}h \in I \colon \varepsilon_n^{\infty}$, so
    $h \in \left\{ \li{-t \varepsilon_n} f \;\middle|\; t \in \N, \; f \in I \colon \varepsilon_n^{\infty}\right\}$.

    The other inclusion is easy: take $t \in \N, f \in I \colon \varepsilon_n^{\infty}$, then $\li{u \varepsilon_n} f \in I$ for some $u \in \N$, so 
    \[
    \li{-t \varepsilon_n}f = \li{(- t - u) \varepsilon_n}{\big(\li{u \varepsilon_n} f\big)} \in \gen{I}_{\N^{n-1} \times \Z} = \gen{S}_{\N^{n-1} \times \Z}.
    \]
\end{proof}

We now provide an algorithm that, given as input the generators of $I \leq \GNn$, computes the generators of $I \colon \varepsilon_n^{\infty} \leq \GNn$.
Our algorithm is based on Bayer's algorithm for ideal saturation~\cite{Bayer1982} (see also~\cite[p.3]{Berthomieu2022F4SAT}).

\begin{theorem}[Subgroup saturation]\label{thm:Bayer_tile}
    There is an algorithm which, given as input a finite set of generators of a stable subgroup $I \leq \GNn$, computes a finite set of generators of $I \colon \varepsilon_n^{\infty} \leq \GNn$.
\end{theorem}
\begin{proof}
    The overall strategy is to adapt a non-homogeneous version of Bayer's ideal saturation algorithm~\cite{Bayer1982} to the stable subgroup setting.
    As a reminder, given an ideal $I \subseteq \K[X_1, \ldots, X_n]$, the (non-homogeneous) Bayer's algorithm is a four-step procedure for computing a generating set of the saturation ideal $I \colon X_n^{\infty}$: 
    \begin{enumerate}[nosep, label=\arabic*.]
        \item \emph{homogenize} $I$;
        \item compute a Gr\"{o}bner basis $B$ of $I$ with respect to the \emph{graded reverse lexicographic order} where $X_n$ is the smallest variable;
        \item factor out from all polynomials in $B$ the highest possible power of $X_n$ to obtain the generators of $I \colon X_n^{\infty}$;
        \item \emph{dehomogenize} the obtained generators.
    \end{enumerate}
    We now adapt this four-step procedure to the stable subgroup setting.
    \smallskip
    
    \textbf{Step 1.}
    For a tile $f \in \GNn$, its \emph{total degree} is defined as 
    \[
    \max\{\alpha_1 + \cdots + \alpha_n \mid (\alpha_1, \ldots, \alpha_n) \in \supp(f)\} \in \N.
    \]
    For a tile $f$ of total degree $d$, its \emph{homogenization}, denoted by $\widetilde f$, is the following tile in $\GN{n+1}$:
    \[
    \widetilde f(\alpha_0, \alpha_1, \ldots, \alpha_n) =
    \begin{cases}
        f(\alpha_1, \ldots, \alpha_n), & \alpha_0 + \alpha_1 + \cdots + \alpha_n = d; \\
        e, & \text{otherwise}.
    \end{cases}
    \]   
    Intuitively, the homogenization $\widetilde f \in \GN{n+1}$ is the ``reverse-projection'' of the tile $f \in \GNn$ on the hyperplane $\{(\alpha_0, \ldots, \alpha_n) \in \N^{n+1} \mid \sum_{i=0}^n \alpha_i = d\}$.
    We say a tile $h \in \GN{n+1}$ is \emph{homogeneous} if $\alpha_0 + \alpha_1 + \cdots + \alpha_{n}$ is constant for all $(\alpha_0, \alpha_1, \ldots, \alpha_{n}) \in \supp(h)$.
    %An \emph{homogeneous stable subgroup} is a stable subgroup generated by homogeneous tiles, not necessarily of the same total degree.
    
    Let $F = \{f_1, \ldots, f_k\}$ be a generating set for $I \leq \GNn$.
    Define $\widetilde I \leq \GN{n+1}$ to be the stable subgroup generated by $\widetilde F = \{\widetilde f_1, \ldots, \widetilde f_k\}$:
    \[
    \widetilde I \coloneqq \gen{\widetilde f_1, \ldots, \widetilde f_k} \leq \GN{n+1}.
    \]
    To compute the generators for $I \colon \varepsilon_n^{\infty} \leq \GNn$, we will first compute the generators for $\widetilde I \colon \varepsilon_n^{\infty} \leq \GN{n+1}$.
    This is done in Steps 2 and 3.
    \smallskip

    \textbf{Step 2.}
    Let $<_{\grevlex}$ be the \emph{graded reverse lexicographic order} on $\N^{n+1}$, that is, we have $\alpha <_{\grevlex} \beta$ if and only if 
    \[
    \sum_{i=0}^{n} \alpha_i < \sum_{i=0}^{n} \beta_i, \; \text{ or } \; \left(\sum_{i=0}^{n} \alpha_i = \sum_{i=0}^{n} \beta_i \;\text{ and }\; \alpha_{k} > \beta_{k}, \alpha_{k+1} = \beta_{k+1}, \ldots, \alpha_{n} = \beta_{n} \text{ for some } k \right).
    \]
    For example, $(0,0,2) <_{\grevlex} (0,1 ,1) <_{\grevlex} (1, 0 ,1)  <_{\grevlex} (0,2,0) <_{\grevlex} (1,1,0) <_{\grevlex} (2, 0, 0)$.
    
    Compute a standard basis $\widetilde B$ of $\widetilde I$ with respect to $<_{\grevlex}$ using Theorem~\ref{thm:compute_tile_basis}.
    By extracting the homogeneous components of each element of $\widetilde B$ we can suppose $\widetilde B$ contains only homogeneous tiles.
    Indeed, if $b \in \widetilde B \subset \widetilde I$ is not homogeneous, then write $b = b_1 b_2 \cdots b_k$ where $b_1, b_2 \ldots, b_k$ are homogeneous of different total degrees.
    Since $\widetilde I$ is generated by homogeneous tiles, it is easy to see that $b_1, b_2 \ldots, b_k \in \widetilde I$.
    Furthermore, $\LT(b) \in \{\LT(b_1), \ldots, \LT(b_k)\}$, so we can replace $b \in \widetilde B$ with the set $\{b_1, b_2 \ldots, b_k\}$.
    \smallskip
    
   \textbf{Step 3.}
   Now write $\widetilde B = \{\widetilde b_1, \ldots, \widetilde b_k\}$ where each $\widetilde b_i \in \widetilde I$ is homogeneous.
    For each $1 \leq i \leq k$, write
    \begin{equation}\label{eq:def_b_prime}
    \widetilde b_i = \li{t_i\varepsilon_n}b_i', \quad b_i' \in \GN{n+1},
    \end{equation}
    where $t_i \in \N$ is the largest integer such that $\supp(\widetilde b_i) - t_i \varepsilon_n \subseteq \N^{n+1}$.
    In other words, $b'_i$ is the translation of $\widetilde b_i$ in the negative direction of $\varepsilon_n$, such that its support touches the hyperplane $\alpha_n = 0$.
    We claim that $\multideg(b'_i) \in \N^n \times \{0\}$.
    Indeed, suppose on the contrary $\multideg(b'_i) = \beta$ with $\beta_n > 0$.
    Then for every $\alpha \in \supp(b'_i)$, we have $\alpha <_{\grevlex} \beta$.
    Since $b'_i$ is homogeneous, by the definition of $<_{\grevlex}$ we have $\alpha_n \geq \beta_n >0$.
    This means that every $\alpha \in \supp(b'_i)$ satisfies $\alpha_n >0$, a contradiction to the maximality of $t_i$.
    We have thus shown $\multideg(b'_i) \in \N^n \times \{0\}$.

    \smallskip
    \begin{adjustwidth}{3mm}{3mm}
    \textbf{Claim.} The set $B' = \{b'_1, \ldots, b'_k\}$ is a standard basis of $\widetilde I \colon \varepsilon_n^{\infty}$ with respect to $<_{\grevlex}$.
    \smallskip 

    \noindent $\blacktriangleright$
    By definition~\eqref{eq:def_b_prime}, we have $b'_i \in \widetilde I \colon \varepsilon_n^{\infty}$ for all $i = 1, \ldots, k$.
    Therefore it suffices to show that 
    \[
    \LT(\widetilde I \colon \varepsilon_n^{\infty}) \subseteq \gen{\LT(b'_1), \ldots, \LT(b'_k)}.
    \]
    Take $f \in \widetilde I \colon \varepsilon_n^{\infty}$, then $\li{t \varepsilon_n} f \in \widetilde I$ for some $t \in \N$.
    Since $\{\widetilde b_1, \ldots, \widetilde b_k\}$ is a standard basis for $\widetilde I$, we have
    \[
    \LT(\li{t \varepsilon_n} f) \in \gen{\LT(\widetilde b_1), \ldots, \LT(\widetilde b_k)} = \gen{\LT(\li{t_1\varepsilon_n}b_1'), \ldots, \LT(\li{t_k\varepsilon_n}b_k')}.
    \]
    Write
    \begin{equation*}
    \LT(\li{t \varepsilon_n} f) = \li{\alpha_1}\LT(\li{t_{i_1}\varepsilon_n}b_{i_1}') \cdot \cdots \cdot  \li{\alpha_s}\LT(\li{t_{i_s}\varepsilon_n}b_{i_s}'),
    \end{equation*}
    where $\multideg(\LT(\li{t \varepsilon_n} f)) = \multideg(\li{\alpha_j}\LT(\li{t_{i_j}\varepsilon_n}b_{i_j}'))$ for all $j = 1, \ldots, s$.
    That is,
    \begin{equation}\label{eq:deg_n_dir}
    t \varepsilon_n + \multideg(f) = \alpha_j + t_{i_j}\varepsilon_n + \multideg(b_{i_j}').
    \end{equation}
    The last entry of the left hand side vector of Equation~\eqref{eq:deg_n_dir} is at least $t$.
    Since $\multideg(b'_{i_j}) \in \N^n \times \{0\}$, the last entry of $\alpha_j + t_{i_j}\varepsilon_n$ must be at least $t$.
    Therefore
    \[
    \alpha_j + t_{i_j}\varepsilon_n - t \varepsilon_n \in \N^{n+1}
    \]
    for all $j$.
    Consequently,
    \[
    \LT(f) = \li{\alpha_1 + t_{i_1}\varepsilon_n - t \varepsilon_n}\LT(b_{i_1}') \cdot \cdots \cdot \li{\alpha_1 + t_{i_s}\varepsilon_n - t \varepsilon_{n}}\LT(b_{i_s}') \in \gen{\LT(b'_1), \ldots, \LT(b'_k)}.
    \]
    %Also, since translation does not alter leading coefficients, Equation~\eqref{eq:homo_LC} implies
    %\[
    %\LC(f) \in \gen{\LC(b'_1), \ldots, \LC(b'_s)}.
    %\]
    %Therefore
    %\[
    %\LT(f) \in \gen{\LT(b'_1), \ldots, \LT(b'_s)} \subseteq \gen{\LT(b'_1), \ldots, \LT(b'_k)},
    %\]
    We conclude that $\LT(\widetilde I \colon \varepsilon_n^{\infty}) \subseteq \gen{\LT(b'_1), \ldots, \LT(b'_k)}$, so $\{b'_1, \ldots, b'_k\}$ is a standard basis for $\widetilde I \colon \varepsilon_n^{\infty}$.
    \hfill $\blacktriangleleft$
    \end{adjustwidth}
    \smallskip

    \textbf{Step 4.} 
    We now dehomogenize $B' = \{b'_1, \ldots, b'_k\} \subset \widetilde I \colon \varepsilon_n^{\infty}$ to obtain a generating set for $I \colon \varepsilon_n^{\infty}$.
    Recall that $\widetilde I$ is defined as the stable subgroup generated by $\widetilde F = \{\widetilde f_1, \ldots, \widetilde f_k\}$, where $F = \{f_1, \ldots, f_k\}$ is a generating set for $I$.
    For a homogeneous tile $f \in \GN{n+1}$ of total degree $d$, its \emph{dehomogenization} is the tile $\widecheck f \in \GNn$ defined by
    \[
    \widecheck f(\alpha_1, \ldots, \alpha_n) = f(d- \alpha_1 - \cdots - \alpha_n, \alpha_1, \ldots, \alpha_n), \quad \forall (\alpha_1, \ldots, \alpha_n) \in \N^n.
    \]
    %Let $<_{\revlex}$ denote the \emph{reverse lexicographic order} on $\N^n$, defined as $\alpha <_{\revlex} \beta$ if and only if
    %\[
    %\alpha_k < \beta_k, \alpha_{k+1} = \beta_{k+1}, \ldots, \alpha_{n} = \beta_{n} \text{ for some } 1 \leq k \leq n.
    %\]
    %Then for a homogeneous $f \in \GN{n+1}$, the dehomogenization of its leading term $\LT(f)$ with respect to $<_{\grlex}$ is equal to the leading term $\LT(\widecheck{f})$ with respect to $<_{\lex}$.
    Geometrically, the dehomogenization of a tile $f \in \GN{n+1}$ is its projection on the last $n$ coordinates.

    \smallskip
    \begin{adjustwidth}{3mm}{3mm}
    \textbf{Claim.} Let $\{b'_1, \ldots, b'_k\}$ be a generating set of $\widetilde I \colon \varepsilon_n^{\infty} \leq \GN{n+1}$. Then $\{\widecheck{b'_1}, \ldots, \widecheck{b'_k}\}$ is a generating set of $I \colon \varepsilon_n^{\infty} \leq \GNn$.
    \smallskip 

    \noindent $\blacktriangleright$
    First we show that $\widecheck{b'_i} \in I \colon \varepsilon_n^{\infty}$ for $i = 1, \ldots, k$.
    Since $b'_i \in \widetilde I \colon \varepsilon_n^{\infty}$, we have $\li{t \varepsilon_n} b'_i \in \widetilde I$ for some $t \in \N$.
    Since $\widetilde I$ is generated by $\widetilde F = \{\widetilde f_1, \ldots, \widetilde f_k\}$, we can write $\li{t \varepsilon_n} b'_i$ in the form
    \begin{equation}\label{eq:pre_dehomo_b}
    \li{t \varepsilon_n} b'_i = \li{\alpha_1} {\widetilde {f_{i_1}}} \cdots \li{\alpha_s} {\widetilde {f_{i_s}}}.
    \end{equation}
    By only keeping the factors with the same total degree as $b'_i$, we can without loss of generality suppose that every $\li{\alpha_j} {\widetilde {f_{i_j}}}$ and have the same total degree as $b'_i$.
    Dehomogenizing Equation~\eqref{eq:pre_dehomo_b} yields
    \[
    \li{t \varepsilon_n} {\widecheck{b'_i}} = \li{\pi(\alpha_1)} f_{i_1} \cdots \li{\pi(\alpha_s)} f_{i_s}.
    \]
    where $\pi \colon \N^{n+1} \rightarrow \N^n$ is the projection on the last $n$ coordinates.
    Therefore $\widecheck{b'_i} \in I \colon \varepsilon_n^{\infty}$ for $i = 1, \ldots, k$.
    
    For the other inclusion, we need to show
    \[
    I \colon \varepsilon_n^{\infty} \subseteq \gen{\widecheck{b'_1}, \ldots, \widecheck{b'_k}}.
    \]
    Take any $f \in I \colon \varepsilon_n^{\infty}$, we have $\li{t \varepsilon_n} f \in I$ for some $t \in \N$.
    Since $I$ is generated by $F = \{f_1, \ldots, f_k\}$, we can write $\li{t \varepsilon_n} f$ in the form
    \begin{equation*}
    \li{t \varepsilon_n}{f} = \li{\alpha_1}{}{f_{i_1}} \cdots \li{\alpha_s}{}{f_{i_s}}.
    \end{equation*}
    By homogenizing\footnote{We homogenize ``at degree $D$'', where $D$ is an integer larger than the total degree of $\li{t \varepsilon_n}{f}$ and all $\li{\alpha_j}{}{f_{i_j}}$. 
    That is, we replace a tile $h$ by the tile $\widetilde h^{(D)}$ with 
    \[
    \widetilde h^{(D)}(\alpha_0, \alpha_1, \ldots, \alpha_n) =
    \begin{cases}
        h(\alpha_1, \ldots, \alpha_n), & \alpha_0 + \alpha_1 + \cdots + \alpha_n = D; \\
        e, & \text{otherwise}.
    \end{cases}
    \]
    In particular, $\li{t \varepsilon_n}{f}$ becomes $\li{t_0 \varepsilon_0 + t \varepsilon_n}{\widetilde f}$, where $t_0$ is the difference between the total degree of $\li{t \varepsilon_n}{f}$ and $D$.
    Similarly, $\li{\alpha_j}{}{f_{i_j}}$ becomes $\li{t_j \varepsilon_0 + \alpha_j}{\widetilde{f_{i_j}}}$, where $t_j$ is the difference between the total degree of $\li{\alpha_j}{}{f_{i_j}}$ and $D$.
    } we get 
    \[
    \li{t_0 \varepsilon_0 + t \varepsilon_n}{\widetilde f} = \li{t_1 \varepsilon_0 + \alpha_1}{\widetilde{f_{i_1}}} \cdots \li{t_s \varepsilon_0 + \alpha_s}{\widetilde{f_{i_s}}} \in \widetilde I,
    \]
    for some $t_0, t_1, \ldots, t_s \in \N$.
    Therefore $\li{t_0 \varepsilon_0}{\widetilde f} \in \widetilde I \colon \varepsilon_n^{\infty}$.
    Since $\{b'_1, \ldots, b'_k\}$ is a generating set for $\widetilde I \colon \varepsilon_n^{\infty}$, we can write $\li{t_0 \varepsilon_0}{\widetilde f}$ in the form
    \begin{equation}\label{eq:pre_dehomo}
    \li{t_0 \varepsilon_0}{\widetilde f} = \li{\beta_1} b'_{i_1} \cdots \li{\beta_r} b'_{i_r}.
    \end{equation}
    Again we can without loss of generality suppose every $\li{\beta_j} b'_{i_j}$ have the same total degree as $\li{t_0 \varepsilon_0}{\widetilde f}$.
    Dehomogenizing Equation~\eqref{eq:pre_dehomo} yields
    \[
    f = \li{\pi(\beta_1)}{\widecheck{b'_1}} \cdots \li{\pi(\beta_r)} {\widecheck{b'_r}}.
    \]
    We conclude that $f \in \gen{\widecheck{b'_1}, \ldots, \widecheck{b'_k}}$ and hence $I \colon \varepsilon_n^{\infty} = \gen{\widecheck{b'_1}, \ldots, \widecheck{b'_k}}$.
    \hfill $\blacktriangleleft$
    \end{adjustwidth}
    \smallskip

    \noindent Combining the four steps, we have thus obtained a finite generating set of $I \colon \varepsilon_n^{\infty} \leq \GNn$.
\end{proof}

Note that Theorem~\ref{thm:Bayer_tile} allows us to compute the saturation $I \colon \varepsilon_i^{\infty} \leq \GNn$ for any $i = 1, \ldots, n$, by simply permuting the indices.
More generally, for $0 \leq m \leq n$, we can define the stable subgroup
\begin{align*}
    I \colon (\varepsilon_{1}, \ldots, \varepsilon_m)^{\infty} & \coloneqq \left\{f \in \GNn \;\middle|\; \exists t_1, \ldots, t_m\in \N \text{ such that } \li{(t_1 \varepsilon_1 + \cdots + t_m \varepsilon_m)}f \in I \right\} \\
    & = ((I \colon \varepsilon_1^{\infty}) \colon \varepsilon_2^{\infty}) \colon \cdots \colon \varepsilon_m^{\infty},
\end{align*}
and use Theorem~\ref{thm:Bayer_tile} iteratively to compute a finite generating set for it.
Furthermore, by applying Lemma~\ref{lem:negative_saturation} iteratively, we obtain:

\begin{observation}\label{obs:negative_saturation_d}
    Let $S \subseteq \GNn$ and $I = \gen{S}$.
    Then for $0 \leq m \leq n$, we have
    \[
    \gen{S}_{\Z^{m} \times \N^{n-m}} = \left\{ \li{-t_1 \varepsilon_1 - \cdots - t_m \varepsilon_m} f \;\middle|\; t_1, \ldots, t_m \in \N, \; f \in I \colon (\varepsilon_{1}, \ldots, \varepsilon_m)^{\infty}\right\} \leq G^{(\Z^m \times \N^{n-m})}.
    \]
\end{observation}

Hence, to describe the elements in $\gen{S}_{\Z^{m} \times \N^{n-m}}$, it suffices to compute $I \colon (\varepsilon_{1}, \ldots, \varepsilon_m)^{\infty} \leq \GNn$.

\subsection{Variable elimination}\label{subsec:eliminiation}
The next procedure we provide is an algorithm for \emph{variable elimination}.
Namely, for a stable subgroup $I \leq \GNn$ and an integer $d \geq 1$, we want to characterize all elements of $I$ supported in $[0, d-1] \times \N^{n-1}$:
\[
I_{[0, d-1] \times \N^{n-1}} \coloneqq \big\{f \in I \;\big|\; \supp(f) \subseteq [0, d-1] \times \N^{n-1} \big\} \leq G^{([0, d-1] \times \N^{n-1})}.
\]
Since $G^{[0, d-1]}$ is the direct product $G^d$, there is a natural isomorphism
\[
G^{([0, d-1] \times \N^{n-1})} \cong (G^d)^{(\N^{n-1})}.
\]
This is explicitly given by $f \mapsto \hat f$, where 
\[
\hat f(\alpha) = \big(f(0, \alpha), \ldots, f(d-1, \alpha) \big) \in G^d, \quad \alpha \in \N^{n-1}.
\]
Under this isomorphism, $I_{[0, d-1] \times \N^{n-1}}$ becomes an $\N^{n-1}$-shift-stable subgroup of $(G^d)^{(\N^{n-1})}$.
The following theorem gives an algorithm that computes the generators for $I_{[0, d-1] \times \N^{n-1}} \leq (G^d)^{(\N^{n-1})}$.
Recall that the lexicographic monomial order $<_{\lex}$ on $\N^n$ is defined as $\alpha <_{\lex} \beta$ if and only if $
\alpha_1 = \beta_1, \ldots, \alpha_{k-1} = \beta_{k-1}, \alpha_k < \beta_k,
$ for some $1 \leq k \leq n$.

\begin{theorem}[Variable elimination]\label{thm:elimination_tile}
    Let $B$ be a standard basis for $I \leq \GNn$ with respect to the monomial order $<_{\lex}$, and let $d$ be a positive integer.
    Then the $\N^{n-1}$-shift-stable subgroup 
    \[
    I_{[0, d-1] \times \N^{n-1}} \leq (G^d)^{(\N^{n-1})}
    \]
    is generated by the finite set
    \[
    B_{[0, d-1]} \coloneqq \big\{ \li{t \varepsilon_1}b \;\big|\; b \in B,\; t < d,\; \supp(\li{t \varepsilon_1}b) \subseteq [0, d-1] \times \N^{n-1} \big\}.
    \]
\end{theorem}
\begin{proof}
    The proof is similar to the classical setting of ideals, cf.~\cite[Chapter~3.1, Theorem~2]{CoxLittleOShea2015}.
    
    By definition we have 
    $
    B_{[0, d-1]} \subseteq I_{\N^{n-1} \times [0, d-1]},
    $
    so it suffices to prove $I_{[0, d-1] \times \N^{n-1}} \subseteq \gen{B_{[0, d-1]}}$.
    Take $f \in I_{[0, d-1] \times \N^{n-1}}$.
    Since $f \in I$ and $B$ is a standard basis for $I$, the reduction of $f$ by $B$ is trivial (Theorem~\ref{thm:membership}).
    That is, we can write $f = \li{\alpha_1}b_1 \cdots \li{\alpha_k}b_k$, where $b_i \in B^{\pm}$ and
    \[
    \multideg(\li{\alpha_i}b_i) \leq_{\lex} \multideg(f).
    \]
    By the definition of $<_{\lex}$ we have $\beta \leq_{\lex} \gamma \implies \beta_1 \leq \gamma_1$.
    Since $\supp(f) \subseteq [0, d-1] \times \N^{n-1}$, we have
    \[
    \supp(\li{\alpha_i}b_i) \subseteq [0, d-1] \times \N^{n-1}.
    \]
    Write $\alpha_i = t \varepsilon_1 + \alpha$ with $t \in \N, \alpha \in \{0\} \times \N^{n-1}$.
    Then $\li{\alpha_i}b_i = \li{\alpha}{\big(\li{t \varepsilon_1}b_i\big)}$ where $\li{t \varepsilon_1}b_i \in B_{[0, d-1]}^{\pm}$.
    Hence $f = \li{\alpha_1}b_1 \cdots \li{\alpha_k}b_k$ is in the $\N^{n-1}$-shift-stable subgroup of $(G^d)^{(\N^{n-1})}$ generated by $B_{[0, d-1]}$.
    We conclude that $I_{[0, d-1] \times \N^{n-1}} \subseteq \gen{B_{[0, d-1]}}$.
\end{proof}

Note that by permuting the indices, Theorem~\ref{thm:elimination_tile} also allows us to eliminate the variable $\varepsilon_n$ instead of $\varepsilon_1$.
That is, we can also compute the generators of 
\[
I_{\N^{n-1} \times [0, d-1]} \coloneqq \big\{f \in I \;\big|\; \supp(f) \subseteq \N^{n-1} \times [0, d-1] \big\} \leq G^{(\N^{n-1} \times [0, d-1])} \cong (G^d)^{(\N^{n-1})}.
\]

\subsection{Nested subgroup membership}\label{subsec:nested}
Finally, we define the problem of \emph{nested subgroup membership} in $\GZn$ and provide an algorithmic solution.
This problem will be the key to the resolution of Subgroup Membership in $G \wr \Z^n$ in the next section.
%In this next section, we will reduce the Subgroup Membership problem in $G \wr \Z^n$ to nested subgroup membership.

Informally, a \emph{nested subgroup} of $\GZn$ is a subgroup generated by different sets of tiles in different subspaces and halfspaces.
Formally, let $B_0, B_{1, +}, B_{1, -} \ldots, B_{n, +}, B_{n, -}$ be $(2n+1)$ sets of tiles in $\GZn$.
The \emph{nested subgroup}
\begin{equation}\label{eq:nested_def}
\gen{B_0 \,\middle|\, B_{1, +}, B_{1, -} \,\middle|\, \cdots \,\middle|\, B_{n, +}, B_{n, -}}
\end{equation}
is the subgroup of $\GZn$ generated by the set
\begin{multline*}
\Big\{b \;\Big|\; b \in B_0 \Big\} \cup \bigcup_{m=1}^n \Big\{ \li{\alpha}b \;\Big|\; b \in B_{m, +},\; \alpha \in \Z^{m-1} \times \N \times \{0\}^{n-m} \Big\} \\ \cup \bigcup_{m=1}^n \Big\{ \li{\alpha}b \;\Big|\; b \in B_{m, -},\; \alpha \in \Z^{m-1} \times (-\N) \times \{0\}^{n-m} \Big\}.
\end{multline*}
In other words, elements in the nested subgroup~\eqref{eq:nested_def} are products
\begin{equation*}
\li{\alpha_1}f_1 \cdot \li{\alpha_2}f_2 \cdot \cdots \cdot \li{\alpha_k}f_k,
\end{equation*}
where $k \in \N$, and
\begin{enumerate}[nosep, label=(\roman*)]
    \item $f_1, \ldots, f_k \in B_0^{\pm} \cup B_{1, +}^{\pm} \cup B_{1, -}^{\pm} \cup \cdots \cup B_{n, +}^{\pm} \cup B_{n, -}^{\pm}$,
    \item if $f_i \in B_0^{\pm}$, then $\alpha_i = 0^n$.
    \item if $f_i \in B_{m, +}^{\pm}$, then $\alpha_i \in \Z^{m-1} \times \N \times \{0\}^{n-m}$,
    \item if $f_i \in B_{m, -}^{\pm}$, then $\alpha_i \in \Z^{m-1} \times (-\N) \times \{0\}^{n-m}$.
\end{enumerate}
\smallskip
That is, for $m = 1, \ldots, n$, the tiles in $B_{m, +}$ can be translated by any integer in the first $m-1$ coordinates, by non-negative integers in the $m$-th coordinate, and they cannot be translated in the last $n-m$ coordinates.
Tiles in $B_{m, -}$ are similar, except they can only be translated by non-positive integers in the $m$-th coordinate.
See Figure~\ref{fig:nested} for an illustration.

\begin{figure}[h!]
    \centering
    \includegraphics[width=0.75\textwidth,height=0.75\textheight,keepaspectratio, trim={2.91cm 1.38cm 2.68cm 1.07cm},clip]{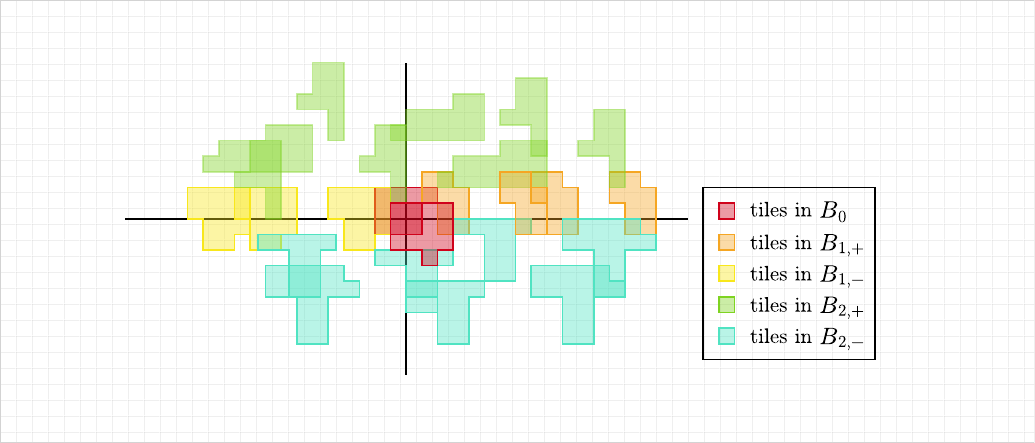}
    \caption{An element in the nested subgroup $\gen{B_0 \,\middle|\, B_{1, +}, B_{1, -} \,\middle|\, B_{2, +}, B_{2, -}} \leq G^{(\Z^2)}$.}
    \label{fig:nested}
\end{figure}

The \emph{nested subgroup membership} problem is the following: given finite sets $B_0, B_{1, +}, B_{1, -}, \ldots,$ $B_{n, +}, B_{n, -} \subset \GZn$ and $f \in \GZn$, decide whether $f \in \gen{B_0 \,\middle|\, B_{1, +}, B_{1, -} \,\middle|\, \cdots \,\middle|\, B_{n, +}, B_{n, -}}$.
In this section we give an algorithmic solution to this problem.
First we describe the overall idea in the abelian setting with $G = \Z$.
Then, we will generalize the idea to finite non-abelian $G$ and give a detailed construction of the algorithm.
In both settings, we will combine saturation and variable elimination, and use induction on the dimension $n$.

\paragraph{Nested subgroup membership for abelian $G$.}
We sketch an algorithm for the nested subgroup membership problem when $G$ is the abelian group $(\Z, +)$.
In this case, we consider $\GZn$ as the additive group of the Laurent polynomial ring $\Z[X_1^{\pm}, \ldots, X_{n}^{\pm}]$:
\begin{align*}
\GZn = \Z^{(\Z^n)} & \xrightarrow{\sim} (\Z[X_1^{\pm}, \ldots, X_{n}^{\pm}], +), \\
f & \mapsto \sum_{\alpha \in \Z^n} f(\alpha) X^{\alpha}.
\end{align*}
Under this identification, the nested subgroup 
\[
\gen{B_0 \,\middle|\, B_{1, +}, B_{1, -} \,\middle|\, \cdots \,\middle|\, B_{n, +}, B_{n, -}} \leq \Z[X_1^{\pm}, \ldots, X_{n}^{\pm}]
\]
is the set of sums of the form 
\begin{equation}\label{eq:nested_classic}
f_0 b_0 + \sum_{m = 1}^n \sum_{i}f_{m, +, i} b_{m, +, i} + \sum_{m = 1}^n \sum_{i} f_{m, -, i} b_{m, -, i},
\end{equation}
where 
\begin{multline*}
f_0 \in \Z,\; f_{m, +, i} \in \Z[X_1^{\pm}, \ldots, X_{m-1}^{\pm}, X_m], \; f_{m, -, i} \in \Z[X_1^{\pm}, \ldots, X_{m-1}^{\pm}, X_m^{-1}], \\
b_0 \in B_0,\; b_{m, +, i} \in B_{m, +},\; b_{m, -, i} \in B_{m, -}.
\end{multline*}
%The membership problem in this set can be solved by ideal saturation and variable elimination.
Given $f \in \Z[X_1^{\pm}, \ldots, X_{n}^{\pm}]$, we now decide whether $f \in \gen{B_0 \,\middle|\, B_{1, +}, B_{1, -} \,\middle|\, \cdots \,\middle|\, B_{n, +}, B_{n, -}}$.
Take a positive integer $d$ such that
\[
\supp(h) \subseteq \Z^{n-1} \times [-d, d], \quad  \text{ for all } h \in \{f\} \cup B_0 \cup B_{1, +} \cup \cdots \cup B_{n, -}.
\]
Thus, if $f$ can be written as a sum~\eqref{eq:nested_classic}, then we must have
\[
\supp\left(\sum_{i}f_{n, +, i} b_{n, +, i}\right) \subseteq \Z^{n-1} \times [-d, d], \quad \supp\left(\sum_{i}f_{n, -, i} b_{n, -, i}\right) \subseteq \Z^{n-1} \times [-d, d].
\]
The set
\[
\left\{\sum_{i}f_{n, +, i} b_{n, +, i} \;\middle|\; f_{n, +, i} \in \Z[X_1^{\pm}, \ldots, X_{n-1}^{\pm}, X_n], b_{n, +, i} \in B_{n, +}\right\} \cap \big\{h \;\big|\; \supp(h) \subseteq \Z^{n-1} \times [-d, d] \big\}
\]
is a $\Z[X_1^{\pm}, \ldots, X_{n-1}^{\pm}]$-module, and one can compute a finite generating set $B'_{n, +}$ for it using saturation and variable elimination (cf.~Observation~\ref{obs:negative_saturation_d} and Theorem~\ref{thm:elimination_tile}).
Similarly, one can define and compute $B'_{n, -}$.
Thus, $f$ can be written as a sum~\eqref{eq:nested_classic} if and only if it can be written as
\[
f = f_0 b_0 + \sum_{m = 1}^{n-1} \sum_{i}f_{m, +, i} b_{m, +, i} + \sum_{m = 1}^{n-1} \sum_{i} f_{m, -, i} b_{m, -, i} + \sum_{b \in B'_{n, +}} f_{b, +} b + \sum_{b \in B'_{n, -}} f_{b, -} b,
\]
where $f_{b, +}, f_{b, -} \in \Z[X_1^{\pm}, \ldots, X_{n-1}^{\pm}]$.
Consequently,
\begin{gather*}
f \in \gen{B_0 \,\middle|\, B_{1, +}, B_{1, -} \,\middle|\, \cdots \,\middle|\, B_{n, +}, B_{n, -}} \\
\iff \\
f \in \gen{B_0 \,\middle|\, B_{1, +}, B_{1, -} \,\middle|\, \cdots \,\middle|\, B_{n-1, +} \cup (B'_{n, +} \cup B'_{n, -}), B_{n-1, -} \cup (B'_{n, +} \cup B'_{n, -}) \,\middle|\, \emptyset, \emptyset}.
\end{gather*}
We can thus obtain an algorithm by induction on the dimension $n$.

In what follows, for $0 \leq m < n$, we will write $\gen{B_0 \mid B_{1, +}, B_{1, -} \mid \cdots \mid B_{m, +}, B_{m, -}}$ for 
\[
\gen{B_0 \mid B_{1, +}, B_{1, -} \mid \cdots \mid B_{m, +}, B_{m, -} \mid \emptyset, \ldots, \emptyset}.
\]

\paragraph{Nested subgroup membership for arbitrary finite $G$.}
The finite non-abelian setting is more subtle: as addition is replaced by a non-abelian operation, one can no longer write elements in the nested subgroup as a sum~\eqref{eq:nested_classic}.
We need to consider the order in which the tiles are applied.
Therefore we need to extend the abelian setting by taking into account conjugation of the tiles.

\begin{theorem}[Nested subgroup membership]\label{thm:nested_tile}
    There is an algorithm which, given as input a finite group $G$, finite sets of tiles $B_0, B_{1, +}, B_{1, -} \ldots, B_{n, +}, B_{n, -}$ in $\GZn$, as well as $f \in \GZn$, decide whether 
    \[
    f \in \gen{B_0 \,\middle|\, B_{1, +}, B_{1, -} \,\middle|\, \cdots \,\middle|\, B_{n, +}, B_{n, -}}.
    \]
\end{theorem}
\begin{proof}
    We use induction on the dimension $n$.
    The base case $n = 0$ is equivalent to determining whether $f$ is in the subgroup of $G$ generated by $B_0$, and is decidable by enumerating all elements of this finite subgroup.
    For the induction step, suppose we have an algorithm for dimension $n-1$, we now construct an algorithm for dimension $n$.
    \smallskip

    \textbf{Step 1.}
    The first observation is that, by enlarging the sets $B_0, B_{1, +}, B_{1, -}, \ldots, B_{n-1, +}, B_{n-1, -}$, we can without loss of generality suppose that tiles in $B_{n, +}$ are supported in the orthant
    \begin{equation}\label{eq:orthant_pos}
    \N^{n-1} \times [1, \infty),
    \end{equation}
    and that tiles in $B_{n, -}$ are supported in the orthant
    \begin{equation}\label{eq:orthant_neg}
    \N^{n-1} \times (-\infty, -1].
    \end{equation}
    Indeed, let $\beta = (\beta_1, \ldots, \beta_n) \in \N^n$ be a vector such that 
    \[
    \supp(\li{\beta}b) \subseteq \N^{n-1} \times [1, \infty) \quad \text{ for every } b \in B_{n, +}.
    \]
    Let $\beta' = (\beta'_1, \ldots, \beta'_n) \in \N^{n-1} \times (-\N)$ be a vector such that 
    \[
    \supp(\li{\beta'}b) \subseteq \N^{n-1} \times (-\infty, -1] \quad \text{ for every } b \in B_{n, -}.
    \]
    We do the following:
    \begin{enumerate}[nosep, label=(\roman*)]
        \item Replace every tile $b$ in $B_{n, +}$ by its translation $\li{\beta}b$.
        \item Replace every tile $b$ in $B_{n, -}$ by its translation $\li{\beta'}b$.
        \item Add the tiles
        $
        \li{t \varepsilon_n}b, \; 0 \leq t < \beta_n,\; b \in B_{n, +}
        $,
        and the tiles
        $
        \li{t \varepsilon_n}b, \; \beta'_n < t \leq 0,\; b \in B_{n, -}
        $,
        into each of the remaining sets $B_0, B_{1, +}, B_{1, -}, \ldots, B_{n-1, +}, B_{n-1, -}$.
    \end{enumerate}
    \smallskip
    Doing these simultaneously, we do not change the nested subgroup $\gen{B_0 \,\middle|\, B_{1, +}, B_{1, -} \,\middle|\, \cdots \,\middle|\, B_{n, +}, B_{n, -}}$.
    As a result, we have $\supp(b) \subseteq \N^{n-1} \times [1, \infty)$ for all $b \in B_{n, +}$, and $\supp(b) \subseteq \N^{n-1} \times (-\infty, -1]$ for all $b \in B_{n, -}$.
    In particular, we have
    \[
    \supp(b_{+}) \cap \supp(b_{-}) = \emptyset
    \]
    for all $b_+ \in B_{n, +}$ and $b_- \in B_{n, -}$.
    \smallskip

    \textbf{Step 2.}
    With the assumption of Step~1, we construct an algorithm for nested subgroup membership by induction on $n$.
    Every element in $\gen{B_0 \,\middle|\, B_{1, +}, B_{1, -} \,\middle|\, \cdots \,\middle|\, B_{n, +}, B_{n, -}}$ can be written as
    \begin{equation}\label{eq:nested_f_product}
    t_0 u_1 v_1 t_1 u_2 v_2 t_2 \cdots u_s v_s t_s, 
    \end{equation}
    where $s \in \N$, and for each $i$,
    \begin{enumerate}[nosep, label=(\roman*)]
        \item $t_i \in \gen{B_0 \,\middle|\, B_{1, +}, B_{1, -} \,\middle|\, \cdots \,\middle|\, B_{n-1, +}, B_{n-1, -}}$.
        \item $u_i$ is either trivial or of the form $\li{\alpha}b,\; b \in B_{n, +}^{\pm},\; \alpha \in \Z^{n-1} \times \N$.
        \item $v_i$ is either trivial or of the form $\li{-\alpha}b,\; b \in B_{n, -}^{\pm},\; \alpha \in \Z^{n-1} \times \N$.
    \end{enumerate}
    \smallskip
    In other words, the factors $u_i$ are the appearances of translations of tiles in $B_{n, +}^{\pm}$, and the factors $v_i$ are the appearances of translations of tiles in $B_{n, -}^{\pm}$.
    The factors $t_i$ are products of the appearances of the remaining tiles.
    
    The key observation is that $u_i$ and $v_j$ commute for all $i, j$: from the previous step we have $\supp(u_i) \cap \supp(v_i) = \emptyset$, so $u_i v_j = v_j u_i$.
    In the product~\eqref{eq:nested_f_product}, moving the factors $u_i$ and $v_i$ to the left using conjugation, we obtain
    \begin{align}\label{eq:nested_product_grouped}
        & t_0 u_1 v_1 t_1 u_2 v_2 t_2 \cdots u_s v_s t_s \nonumber \\
        =\; & u_1 v_1 u_2 v_2 \cdots u_s v_s \cdot \left(t_0^{u_1 v_1 \cdots u_s v_s} \cdot t_1^{u_2 v_2 \cdots u_s v_s} \cdot \cdots \cdot t_s\right) \nonumber \\
        =\; & \underbrace{(u_1 u_2 \cdots u_s)}_{\in \gen{B_{n, +}}_{\Z^{n-1} \times \N}} \cdot \underbrace{(v_1 v_2 \cdots v_s)}_{\in \gen{B_{n, -}}_{\Z^{n-1} \times -\N}} \cdot \underbrace{\left(t_0^{u_1 \cdots u_s v_1 \cdots v_s} \cdot t_1^{u_2 \cdots u_s v_2 \cdots v_s} \cdot \cdots \cdot t_s\right)}_{\in \gen{B_0^I \,\middle|\, B_{1, +}^I, B_{1, -}^I \,\middle|\, \cdots \,\middle|\, B_{n-1, +}^I, B_{n-1, -}^I}}.
    \end{align}
    Here, $B^I$ denotes the set $\{b^h \mid b \in B, h \in I\}$, where
    \[
    I \coloneqq \gen{B_{n, +}}_{\Z^{n-1} \times \N} \cdot \gen{B_{n, -}}_{\Z^{n-1} \times -\N} = \left\{uv \;\middle|\; u \in \gen{B_{n, +}}_{\Z^{n-1} \times \N}, v \in \gen{B_{n, -}}_{\Z^{n-1} \times -\N} \right\}.
    \]
    %Therefore, we have $f \in \gen{B_0 \,\middle|\, B_{1, +}, B_{1, -} \,\middle|\, \cdots \,\middle|\, B_{n, +}, B_{n, -}}$ if and only if $f$ can be written as a product~\eqref{eq:nested_product_grouped}.
    
    Given $f \in \GZn$, we now decide whether $f \in \gen{B_0 \,\middle|\, B_{1, +}, B_{1, -} \,\middle|\, \cdots \,\middle|\, B_{n, +}, B_{n, -}}$.
    Let $d$ be a positive integer such that
    \[
    \supp(h) \subseteq \Z^{n-1} \times [-d, d] \quad \text{ for all } h \in \{f\} \cup B_0 \cup B_{1, +} \cup \cdots \cup B_{n, -}.
    \]
    If $f$ can be written in the form~\eqref{eq:nested_product_grouped}, then we must have 
    \[
    \supp(u_1 u_2 \cdots u_s) \subseteq \Z^{n-1} \times [1, d], \quad \text{ and } \quad \supp(v_1 v_2 \cdots v_s) \subseteq \Z^{n-1} \times [-d, -1].
    \]
    By Observation~\ref{obs:negative_saturation_d}, we have 
    \begin{equation}\label{eq:sat_g_pos}
    \gen{B_{n, +}}_{\Z^{n-1} \times \N} = \left\{\li{\alpha} u \;\middle|\; \alpha \in \Z^{n-1} \times \{0\},\; u \in \gen{B_{n, +}} \colon (\varepsilon_1, \ldots, \varepsilon_{n-1})^{\infty} \right\}.
    \end{equation}
    By Theorem~\ref{thm:Bayer_tile} (subgroup saturation), we can compute a finite set of generators for the stable subgroup
    \[
    \gen{B_{n, +}} \colon (\varepsilon_1, \ldots, \varepsilon_{n-1})^{\infty} \leq \GNn.
    \]
    Then by Theorem~\ref{thm:elimination_tile} (variable elimination), we can compute a finite set of generators $B_> \subseteq G^{(\N^{n-1} \times [1, d])}$ for the intersection
    \[
    \big(\gen{B_{n, +}} \colon (\varepsilon_1, \ldots, \varepsilon_{n-1})^{\infty}\big) \cap G^{(\N^{n-1} \times [1, d])},
    \]
    considered as a $\N^{n-1}$-shift-stable subgroup of $(G^d)^{(\N^{n-1})}$.
    That is,
    \[
    \gen{B_>}_{\N^{n-1}} = \big(\gen{B_{n, +}} \colon (\varepsilon_1, \ldots, \varepsilon_{n-1})^{\infty}\big) \cap G^{(\N^{n-1} \times [1, d])}.
    \]
    Translating the above equation by $\Z^{n-1} = \Z^{n-1} \times \{0\}$ and applying Equation~\eqref{eq:sat_g_pos}, we obtain
    \begin{equation*}%\label{eq:nested_sat_pos}
    \gen{B_>}_{\Z^{n-1}} = \gen{B_{n, +}}_{\Z^{n-1} \times \N} \cap G^{(\Z^{n-1} \times [1, d])}.
    \end{equation*}

    Similarly, we can compute a finite set $B_< \subseteq G^{(\N^{n-1} \times [-d, -1])}$, such that
    \begin{equation*}%\label{eq:nested_sat_neg}
    \gen{B_<}_{\Z^{n-1}} = \gen{B_{n, -}}_{\Z^{n-1} \times -\N} \cap G^{(\Z^{n-1} \times [-d, -1])},
    \end{equation*}
    by simply reflecting over the $n$-th coordinate.
    This naturally leads to the following claim:

    \smallskip
    \begin{adjustwidth}{3mm}{3mm}
    \textbf{Claim.} We have $f \in \gen{B_0 \,\middle|\, B_{1, +}, B_{1, -} \,\middle|\, \cdots \,\middle|\, B_{n, +}, B_{n, -}}$, if and only if $f$ is in the following nested subgroup of $G^{(\Z^{n-1} \times [-d, d])} \cong (G^{2d + 1})^{(\Z^{n-1})}$:
    \begin{equation}\label{eq:nested_d_lowdim}
    \gen{B_0^{I} \,\middle|\, B_{1, +}^{I}, B_{1, -}^{I} \,\middle|\, \cdots \,\middle|\, B_{n-2, +}^{I}, B_{n-2, -}^{I} \,\middle|\, B_{n-1, +}^{I} \cup B_> \cup B_<, B_{n-1, -}^{I} \cup B_> \cup B_<}.
    \end{equation}
    Here, $B^I$ denotes the set $\{b^h \mid b \in B, h \in I\}$, where 
    \[
    I \coloneqq \gen{B_{n, +}}_{\Z^{n-1} \times \N} \cdot \gen{B_{n, -}}_{\Z^{n-1} \times -\N} = \left\{uv \;\middle|\; u \in \gen{B_{n, +}}_{\Z^{n-1} \times \N}, v \in \gen{B_{n, -}}_{\Z^{n-1} \times -\N} \right\} \leq \GZn.
    \]
    
    \noindent$\blacktriangleright$
    Obviously the nested subgroup~\eqref{eq:nested_d_lowdim} is contained in $\gen{B_0 \,\middle|\, B_{1, +}, B_{1, -} \,\middle|\, \cdots \,\middle|\, B_{n, +}, B_{n, -}}$.
    For the opposite implication, suppose $f \in \gen{B_0 \,\middle|\, B_{1, +}, B_{1, -} \,\middle|\, \cdots \,\middle|\, B_{n, +}, B_{n, -}}$.
    Then by Equation~\eqref{eq:nested_product_grouped}, $f$ can be written as a product $uvt$, where 
    \[
    u \in \gen{B_{n, +}}_{\Z^{n-1} \times \N} \cap G^{(\Z^{n-1} \times [1, d])} = \gen{B_>}_{\Z^{n-1}},
    \]
    \[
    v \in \gen{B_{n, -}}_{\Z^{n-1} \times \N} \cap G^{(\Z^{n-1} \times [-d, -1])} = \gen{B_<}_{\Z^{n-1}},
    \]
    and
    \[
    t \in \gen{B_0^I \,\middle|\, B_{1, +}^I, B_{1, -}^I \,\middle|\, \cdots \,\middle|\, B_{n-1, +}^I, B_{n-1, -}^I}.
    \]
    Since $\Z^{n-1} = (\Z^{n-2} \times \N) \cup (\Z^{n-2} \times -\N)$, each $\Z^{n-1}$-translation of a tile in $B_>$ or $B_<$ is in fact a $(\Z^{n-2} \times \N)$-translation or a $(\Z^{n-2} \times -\N)$-translation.
    In other words,
    \[
    \gen{B_>}_{\Z^{n-1}} = \gen{\emptyset \,\middle|\, \emptyset, \emptyset \,\middle|\, \cdots \,\middle|\, B_>, B_>}, \quad \gen{B_<}_{\Z^{n-1}} = \gen{\emptyset \,\middle|\, \emptyset, \emptyset \,\middle|\, \cdots \,\middle|\, B_<, B_<}
    \]
    Therefore we have
    \[
    f = uvt \in \gen{B_0^{I} \,\middle|\, B_{1, +}^{I}, B_{1, -}^{I} \,\middle|\, \cdots \,\middle|\, B_{n-1, +}^{I} \cup B_> \cup B_<, B_{n-1, -}^{I} \cup B_> \cup B_<}.
    \]
    \hfill $\blacktriangleleft$
    \end{adjustwidth}
    \smallskip

    \noindent The above claim reduces the membership problem in the nested subgroup 
    \[
    \gen{B_0 \,\middle|\, B_{1, +}, B_{1, -} \,\middle|\, \cdots \,\middle|\, B_{n, +}, B_{n, -}} \subseteq \GZn
    \]
    in dimension $n$, to the membership problem in the nested subgroup
    \[
    \gen{B_0^{I} \,\middle|\, B_{1, +}^{I}, B_{1, -}^{I} \,\middle|\, \cdots \,\middle|\, B_{n-1, +}^{I} \cup B_> \cup B_<, B_{n-1, -}^{I} \cup B_> \cup B_<} \subseteq (G^{2d+1})^{\Z^{n-1}}
    \]
    in dimension $n-1$.
    We apply the induction hypothesis on $n-1$.
    Note that $G^{2d+1}$ is still a finite group, and the sets $B_0^{I}, B_{1, +}^{I}, B_{1, -}^{I}, \ldots, B_{n-1, -}^{I}$ are finite and computable using the same argument as the first claim in Theorem~\ref{thm:compute_tile_basis}.
    Therefore, we obtain an algorithm for deciding nested subgroup membership in dimension $n$.
\end{proof}

\section{Subgroup Membership in finite-wreath-abelian groups}\label{sec:membership}

In this section we prove decidability of Subgroup Membership in the wreath product $G \wr \Z^n$, by reducing it to nested subgroup membership in $\GZn$.
The gist of the reduction is an elementary but technical application of the Reidemeister-Schreier method~\cite[Chapter~II.4]{LyndonSchupp1977} (see also~\cite{Putman2022}).

Recall that the group $G \wr \Z^n$ consists of elements $(f, \alpha)$ where $f \in \GZn, \alpha \in \Z^n$, and multiplication is defined by
\[
(f, \alpha) \cdot (g, \beta) = (f \cdot \li{\alpha}g,\; \alpha + \beta).
\]
Intuitively, the element $(f, \alpha) \in G \wr \Z^n$ can be visualized as a tile $f$ together with a vector $\alpha$.
Left-multiplication by $(f, \alpha)$ corresponds to translating the current element by $\alpha$ and applying the tile $f$ on top.
See Figures~\ref{fig:wreathf}, \ref{fig:wreathg}, \ref{fig:wreathfg} for illustration.

\begin{figure}[ht!]
    \centering
    \begin{minipage}[t]{.27\textwidth}
        \centering
        \includegraphics[width=0.8\textwidth,height=0.8\textheight,keepaspectratio, trim={6.3cm 1.1cm 5.84cm 1.05cm},clip]{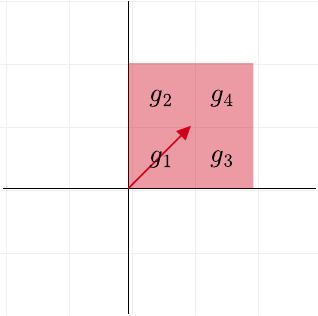}
        \caption{An element $h_1 = (f_1, (1,1)) \in G \wr \Z^2$}
        \label{fig:wreathf}
    \end{minipage}
    \hfill
    \begin{minipage}[t]{.27\textwidth}
        \centering
        \includegraphics[width=0.8\textwidth,height=0.8\textheight,keepaspectratio, trim={6.3cm 1.1cm 5.84cm 1.05cm},clip]{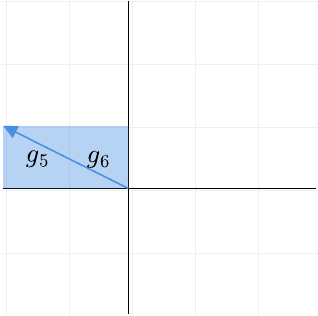}
        \caption{An element $h_2 = (f_2, (-2,0)) \in G \wr \Z^2$}
        \label{fig:wreathg}
    \end{minipage}
    \hfill
    \begin{minipage}[t]{0.27\textwidth}
        \centering
        \includegraphics[width=0.8\textwidth,height=0.8\textheight,keepaspectratio, trim={6.3cm 1.1cm 5.84cm 1.05cm},clip]{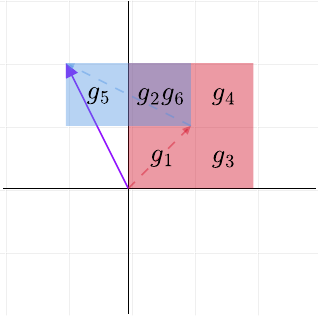}
        \caption{The product $h_1 h_2 = (f_1 \cdot \li{(1,1)}f_2, (-1,2))$.}
        \label{fig:wreathfg}
    \end{minipage}
\end{figure}

The group $G \wr \Z^n$ naturally contains the subgroup 
\[
\GZn \leq G \wr \Z^n,
\]
consisting of all elements of the form $(f, 0),\; f \in \GZn$.
As in the previous section, let 
\[
\varepsilon_1 = (1, 0, \ldots, 0),\; \ldots,\; \varepsilon_n = (0, \ldots, 0, 1),
\]
denote the canonical basis of $\Z^n$.

For a set of elements $S \subseteq G \wr \Z^n$, denote by $\gen{S}_{\gp}$ the subgroup generated by $S$.
Recall that the Subgroup Membership problem in $G \wr \Z^n$ takes as input a finite set $S \subset G \wr \Z^n$ and an element $h \in G \wr \Z^n$, and asks whether $h \in \gen{S}_{\gp}$.
In this paper, the finite group $G$ and the dimension $n$ can also be taken as part of the input.

The following lemma reduces Subgroup Membership in $G \wr \Z^n$ to a special case where $h \in \GZn$ and where $S$ is of a specific form:

\begin{lemma}\label{lem:subgroup_special}
    Subgroup Membership in $G \wr \Z^n$ reduces to the following decision problem:
    \smallskip
    \begin{adjustwidth}{3mm}{3mm}
    \textbf{Input:} A finite group $\widehat G$, an integer $d \leq n$, and elements $(g_1, \varepsilon_1), \ldots, (g_d, \varepsilon_d), (f_1, 0), \ldots, (f_k, 0)$, $(f, 0)$ in the wreath product $\widehat G \wr \Z^d$.
    
    \noindent \textbf{Question:} whether $(f, 0) \in \gen{(g_1, \varepsilon_1), \ldots, (g_d, \varepsilon_d), (f_1, 0), \ldots, (f_k, 0)}_{\gp}$.
    \end{adjustwidth}
\end{lemma}
\begin{proof}
    Let $h_1, \ldots, h_k, h \in G \wr \Z^n$.
    For $H \coloneqq \gen{h_1, \ldots, h_k}_{\gp}$, we will decide whether $h \in H$.
    Consider the projection map
    \[
    \pi \colon G \wr \Z^n \rightarrow \Z^n, \quad (f, \alpha) \mapsto \alpha.
    \]
    Then $\pi(h_1), \ldots, \pi(h_k)$ generate the group $\pi(H) \leq \Z^n$.
    Let $d \coloneqq \dim\pi(H)$, and let 
    \[
    \alpha_1, \ldots, \alpha_d \in \pi(H)
    \]
    be a $\Z$-basis of $\pi(H)$.
    \smallskip
    
    \textbf{Step 1.} We reduce Subgroup Membership in $G \wr \Z^n$ to the special case where $h = (f, 0)$ and the subgroup generators are of the form $(g_1, \alpha_1), \ldots, (g_d, \alpha_d), (f_1, 0), \ldots, (f_k, 0)$.

    Since the vectors $\alpha_1, \ldots, \alpha_d$ are in $\pi(H) = \sum_{j=1}^k \Z \cdot \pi(h_j)$, we can find $c_{ij} \in \Z, i = 1, \ldots, d; j = 1, \ldots, k$, such that
    \[
    \alpha_i = \sum_{j=1}^k c_{ij} \pi(h_j), \quad \text{ for } i = 1, \ldots, d.
    \]
    Consider the elements
    \[
    (g_i, \alpha_i) \coloneqq \prod_{j=1}^k h_j^{c_{ij}} \in H, \quad i = 1, \ldots, d.
    \]
    Since the vectors $\alpha_1, \ldots, \alpha_d$ generate $\pi(H)$, we can find $r_{ij} \in \Z, i = 1, \ldots, d; j = 1, \ldots, k$, such that
    \[
    \pi(h_j) = \sum_{i=1}^d r_{ij} \alpha_i, \quad \text{ for } j = 1, \ldots, k.
    \]
    Define the elements
    \[
    (f_j, 0) \coloneqq h_j \cdot \prod_{i = 1}^d (g_i, \alpha_i)^{-r_{ij}}, \quad \text{ for } j = 1, \ldots, k.
    \]
    Then
    \begin{align*}
        H & = \gen{h_1, \ldots, h_k}_{\gp} \\
        & = \gen{(g_1, \alpha_1), \ldots, (g_d, \alpha_d), h_1, \ldots, h_k}_{\gp} \\
        & = \gen{(g_1, \alpha_1), \ldots, (g_d, \alpha_d), (f_1, 0), \ldots, (f_k, 0)}_{\gp}.
    \end{align*}

    In order to decide whether $h \in H$, consider two cases.
    If $\pi(h) \notin \pi(H)$, then obviously $h \notin H$.
    If $\pi(h) \in \pi(H)$, write $\pi(h) = \sum_{i=1}^d s_i \alpha_i$.
    Then $h \in H$ if and only if $(f, 0) \coloneqq h \cdot \prod_{i=1}^d (g_i, \alpha_i)^{-s_i}$ is in $H$.
    We have thus reduced the membership problem in $H$ to deciding whether
    \[
    (f, 0) \in \gen{(g_1, \alpha_1), \ldots, (g_d, \alpha_d), (f_1, 0), \ldots, (f_k, 0)}_{\gp}.
    \]

    \textbf{Step 2.} We reduce the above problem to Subgroup Membership in $\tilde G \wr \Z^n$ with finite $\tilde G$, where the generators are of the form $(g'_1, \varepsilon_1), \ldots, (g'_d, \varepsilon_d), (f'_1, 0), \ldots, (f'_k, 0)$.

    Let $H = \gen{(g_1, \alpha_1), \ldots, (g_d, \alpha_d), (f_1, 0), \ldots, (f_k, 0)}_{\gp}$.
    Complete $\alpha_1, \ldots, \alpha_d$ into a basis 
    \[
    \{\alpha_1, \ldots, \alpha_d, \alpha_{d+1}, \ldots, \alpha_n\}
    \]
    of an $n$-dimensional lattice.
    In other words, we find $\alpha_{d+1}, \ldots, \alpha_n \in \Z^n$ such that $\sum_{i = 1}^n \Z \alpha_i$ is a finite index subgroup of $\Z^n$.
    Let $C \subset \Z^n$ be a set of representatives of the cosets of $\sum_{i = 1}^n \Z \alpha_i$ in $\Z^n$.
    We obtain an injective group homomorphism
    \[
    \varphi \colon H \hookrightarrow G^C \wr \Z^n,
    \]
    which sends $(f, \beta) \in G \wr \Z^n$ to 
    \[
    (g, (b_1, \ldots, b_d, 0, \ldots, 0)) \in G^C \wr \Z^n,
    \]
    where $b_1, \ldots, b_d$ are the unique integers such that $\beta = \sum_{i = 1}^d b_i \alpha_i$, and $g \in (G^C)^{(\Z^n)}$ is defined as
    \[
    g(r_1, \ldots, r_n) = \left(f\left(\sum_{i = 1}^n r_i \alpha_i + c \right) \right)_{c \in C} \in G^C, \quad \forall (r_1, \ldots, r_n) \in \Z^n.
    \]
    The map $\varphi$ is injective because every $\beta \in \Z^n$ can be uniquely written as $\beta = \sum_{i = 1}^n r_i \alpha_i + c$ with $r_i \in \Z, c \in C$.
    Therefore, the question of whether $(f, 0) \in H$ reduces to whether $\varphi(f, 0) \in \varphi(H)$, where $\varphi(H) \leq G^C \wr \Z^n$ is generated by 
    \[
    \varphi(g_1, \alpha_1) = (g'_1, \varepsilon_1),\; \ldots,\; \varphi(g_d, \alpha_d) = (g'_d, \varepsilon_d),\; \varphi(f_1, 0) = (f'_1, 0),\; \ldots,\; \varphi(f_k, 0) = (f'_k, 0),
    \]
    for some $g'_1, \ldots, g'_d, f'_1, \ldots, f'_k \in (G^C)^{(\Z^n)}$.
    Let $\tilde G$ denote the finite group $G^C$.
    \smallskip

    \textbf{Step 3.} Finally, let $(f', 0) \coloneqq \varphi(f, 0) \in \tilde G \wr \Z^n$. Let $M \in \N$ be a number such that
    \[
    \supp(g'_i), \supp(f'_j), \supp(f') \subseteq \Z^d \times [-M, M]^{n-d}
    \]
    for all $i = 1, \ldots, d; j = 1, \ldots, k$.
    Then $\varphi(H) = \gen{(g'_1, \varepsilon_1), \ldots, (g'_d, \varepsilon_d), (f'_1, 0), \ldots, (f'_k, 0)}_{\gp} \leq \tilde G \wr \Z^n$ is contained in the subgroup
    \[
    \big(\tilde{G}^{[-M, M]^{n-d}}\big) \wr \Z^d \leq \tilde{G} \wr \Z^n.
    \]
    Let $\widehat G \coloneqq \tilde{G}^{[-M, M]^{n-d}}$, we have thus reduced the problem to deciding whether $(f', 0) \in \widehat{G} \wr \Z^d$ is in the subgroup of $\widehat{G} \wr \Z^d$ generated by $(g'_1, \varepsilon_1), \ldots, (g'_d, \varepsilon_d), (f'_1, 0), \ldots, (f'_k, 0)$.
\end{proof}

Thanks to Lemma~\ref{lem:subgroup_special}, by replacing $\widehat G$ by $G$ and $d$ by $n$, from now on we can consider only the special case of Subgroup Membership where we decide whether $(f, 0) \in G \wr \Z^n$ is in the subgroup 
\[
H = \gen{(g_1, \varepsilon_1), \ldots, (g_n, \varepsilon_n), (f_1, 0), \ldots, (f_k, 0)}_{\gp} \leq G \wr \Z^n.
\]
In what follows, we will often identify a tile $f \in \GZn$ with the group element $(f, 0) \in G \wr \Z^n$.
To decide whether $f = (f, 0) \in G \wr \Z^n$ is in $H$, we need to characterize the elements in $H \cap \GZn$.
Let $[g, h]$ denote the commutator 
\[
[g, h] \coloneqq ghg^{-1}h^{-1}.
\]
The first observation is that the following sets of commutators are in $H \cap \GZn$:
\begin{align*}
C_{++} & \coloneqq \big\{[g_i, \varepsilon_i), (g_j, \varepsilon_j)] \;\big|\; 1 \leq i, j \leq n \big\}, \quad && C_{+-} \coloneqq \big\{[g_i, \varepsilon_i), (g_j, \varepsilon_j)^{-1}] \;\big|\; 1 \leq i, j \leq n \big\}, \\
C_{-+} & \coloneqq \big\{[g_i, \varepsilon_i)^{-1}, (g_j, \varepsilon_j)] \;\big|\; 1 \leq i, j \leq n \big\}, \quad && C_{--} \coloneqq \big\{[g_i, \varepsilon_i)^{-1}, (g_j, \varepsilon_j)^{-1}] \;\big|\; 1 \leq i, j \leq n \big\}.
\end{align*}
%These can be considered as tiles in $\GZn$.
For $f, t \in G \wr \Z^n$, let $\li{f}t$ denote the left conjugation 
\[
\li{f}t \coloneqq ftf^{-1}. 
\]
Let $e^{\Z^n} \in \GZn$ denote the trivial tile, then for $t \in \GZn$ and $\alpha \in \Z^n$, we have
\[
\li{(e^{\Z^n}, \alpha)}t = \li{\alpha}t \in \GZn.
\]

We now provide a characterization of elements in $H \cap \GZn$.

\begin{proposition}[{Variant of Reidemeister-Schreier, compare~\cite{tomaszewski2003basis, Putman2022}}]\label{prop:standard_form}
    Let $f \in \GZn$.
    Then $f \in H \cap \GZn$ if and only if $f$ can be written as
    \begin{equation}\label{eq:standard_form}
    \li{(g_1, \varepsilon_1)^{z_{11}} \cdots (g_n, \varepsilon_n)^{z_{1n}}}t_1 \cdot \li{(g_1, \varepsilon_1)^{z_{21}} \cdots (g_n, \varepsilon_n)^{z_{2n}}}t_2 \cdot \cdots \cdot \li{(g_1, \varepsilon_1)^{z_{m1}} \cdots (g_n, \varepsilon_n)^{z_{mn}}}t_m,
    \end{equation}
    where $m \in \N$, $t_i \in \{f_1, \ldots, f_k\}^{\pm} \cup C_{++} \cup C_{+-} \cup C_{-+} \cup C_{--}$, and $z_{ij} \in \Z$ for $i = 1, \ldots, m; j = 1, \ldots, n$.
\end{proposition}
\begin{proof}
    The proof is an application of the Reidemeister-Schreier method \cite[Chapter~II.4]{LyndonSchupp1977}.
    See Figure~\ref{fig:canonical} for an illustration of a simple example.
    See also~\cite{tomaszewski2003basis, Putman2022} for similar proofs in the context of the commutator subgroup of free groups.

    \begin{figure}[h!]
    \centering
    \includegraphics[width=0.7\textwidth,height=0.7\textheight,keepaspectratio, trim={2.1cm 1.1cm 1.6cm 0cm},clip]{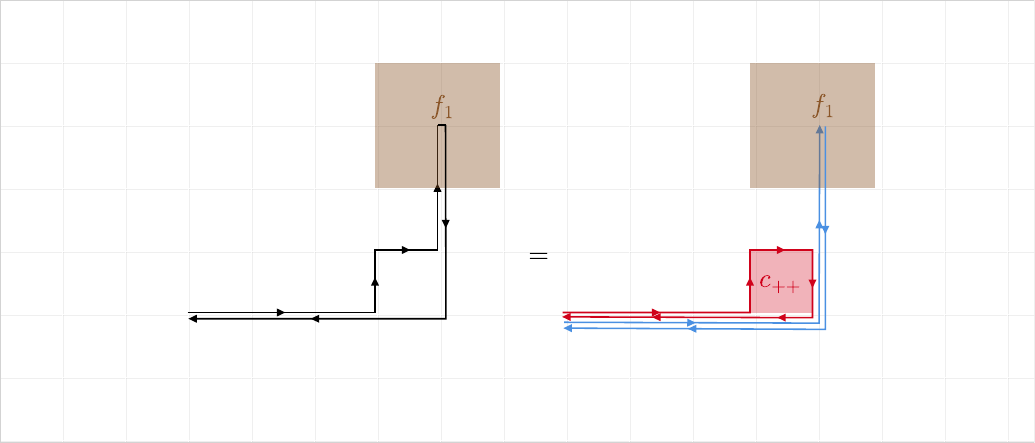}
    \caption{Example for Proposition~\ref{prop:standard_form}, writing an element of $H \cap \GZn$ in the form~\eqref{eq:standard_form}.}
    \label{fig:canonical}
    \end{figure}

    Recall that $H = \gen{(g_1, \varepsilon_1), \ldots, (g_n, \varepsilon_n), (f_1, 0), \ldots, (f_k, 0)}_{\gp}$.
    Since $C_{++}, C_{+-}, C_{-+}, C_{--} \subseteq H \cap \GZn$, any element of the form~\eqref{eq:standard_form} is in $H \cap \GZn$.
    For the opposite implication, let $f \in H \cap \GZn$ and we show it can be written in the form~\eqref{eq:standard_form}.
    We proceed in two steps.
    \smallskip

    \textbf{Step 1.} First we show that $f$ can be written as
    \begin{equation*}
    \li{p_1}t_1 \cdot \li{p_2}t_2 \cdot \cdots \cdot \li{p_m}t_m \cdot p_0,
    \end{equation*}
    where $m \in \N$, $p_i \in \gen{(g_1, \varepsilon_1), \ldots, (g_n, \varepsilon_n)}_{\gp}$ and $t_i \in \{f_1, \ldots, f_k\}^{\pm}$ for all $0 \leq i \leq m$.
    Additionally we have $p_0 \in \GZn$.
    Intuitively, $p_i$ can be considered as a \emph{path} in the $n$-dimensional grid composed of steps $(g_i, \varepsilon_i)^{\pm}$, and the conjugation $\li{p_i}t_i$ corresponds to traversing the path $p_i$, applying the tile $t_i$ at the end of the path, then coming back to the origin by inverting the path.
    Additionally, the path $p_0$ is a cycle.
    See Figure~\ref{fig:cycle} for an illustration.

    \begin{figure}[b]
    \centering
    \includegraphics[width=0.7\textwidth,height=0.7\textheight,keepaspectratio, trim={3.17cm 1.64cm 3.21cm 1.59cm},clip]{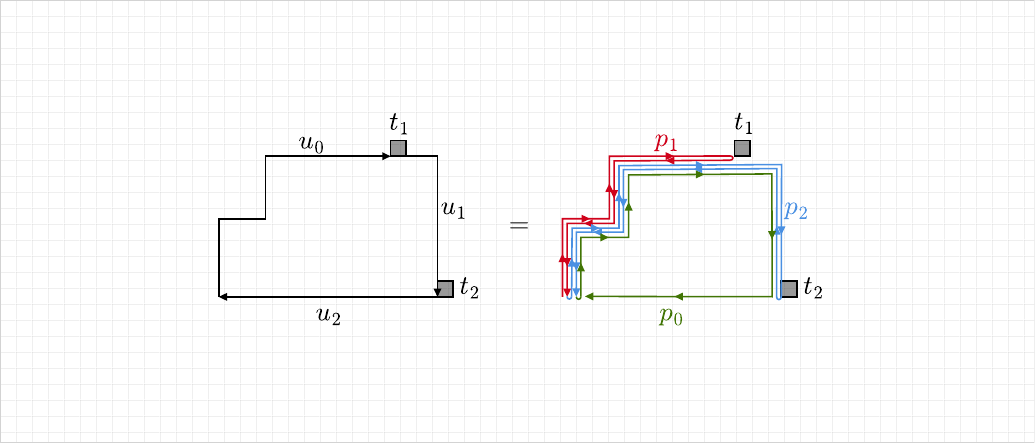}
    \caption{Illustration for Step 1, writing $f = u_0 t_1 u_1 t_2 u_2$, as $\li{p_1}t_1 \cdot \li{p_2}t_2 \cdot p_0$.}
    \label{fig:cycle}
    \end{figure}

    Indeed, since $f$ is in the group $H$ generated by the elements $(g_1, \varepsilon_1), \ldots, (g_n, \varepsilon_n), f_1, \ldots, f_k$, we can write $f$ as a product
    \[
    f = u_0 t_1 u_1 t_2 u_2 \cdots t_m u_m,
    \]
    where $u_i \in \gen{(g_1, \varepsilon_1), \ldots, (g_n, \varepsilon_n)}_{\gp}$ and $t_i \in \{f_1, \ldots, f_k\}^{\pm}$.
    Moving all the factors $u_i$ to the right yields
    \[
    f = \li{u_0}t_1 \cdot \li{u_0 u_1} t_2 \cdot \cdots \cdot \li{u_0 u_1 \cdots u_{m-1}} t_m \cdot u_0 u_1 \cdots u_m.
    \]
    We obtain
    \[
    f = \li{p_1}t_1 \cdot \li{p_2}t_2 \cdot \cdots \cdot \li{p_m}t_m \cdot p_0
    \]
    by taking the elements
    \[
    p_1 = u_0, \; p_2 = u_0 u_1,\; \ldots,\; p_m = u_0 u_1 \cdots u_{m-1},\; \text{ and }\; p_0 = u_0 u_1 \cdots u_m.
    \]
    Then $p_i \in \gen{(g_1, \varepsilon_1), \ldots, (g_n, \varepsilon_n)}_{\gp}$ for all $i$.
    Furthermore, we have $p_0 \in \GZn$ because $f \in \GZn$ and $\li{p_i} t_i \in \GZn$ for $1 \leq i \leq m$.

    \smallskip

    \textbf{Step 2.} For every $p \in \gen{(g_1, \varepsilon_1), \ldots, (g_n, \varepsilon_n)}_{\gp}$, we show that it can be written as
    \begin{equation}\label{eq:p_standard}
    \li{(g_1, \varepsilon_1)^{z_{11}} \cdots (g_n, \varepsilon_n)^{z_{1n}}}c_1 \cdot \li{(g_1, \varepsilon_1)^{z_{21}} \cdots (g_n, \varepsilon_n)^{z_{2n}}}c_2 \cdot \cdots \cdot \li{(g_1, \varepsilon_1)^{z_{s1}} \cdots (g_n, \varepsilon_n)^{z_{sn}}}c_s \cdot (g_1, \varepsilon_1)^{z_{01}} \cdots (g_n, \varepsilon_n)^{z_{0n}},
    \end{equation}
    where $s \in \N$, $c_i \in C_{++} \cup C_{+-} \cup C_{-+} \cup C_{--}$, and $z_{ij} \in \Z$ for all $i, j$.
    See Figure~\ref{fig:decompose} for an illustration.

    \begin{figure}[b]
    \centering
    \includegraphics[width=0.7\textwidth,height=0.7\textheight,keepaspectratio, trim={3.15cm 1.1cm 2.65cm 2.1cm},clip]{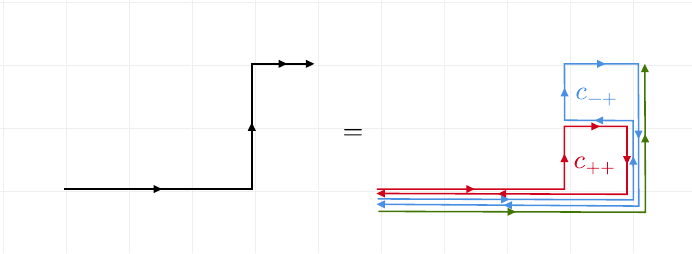}
    \caption{Illustration for Step 2, writing $(g_1, \varepsilon_1)^3 (g_2, \varepsilon_2)^2 (g_1, \varepsilon_1)$ as $\li{(g_1, \varepsilon_1)^3}{}{c_{++}} \cdot \li{(g_1, \varepsilon_1)^4 (g_2, \varepsilon_2)}{}{c_{-+}} \cdot (g_1, \varepsilon_1)^4 (g_2, \varepsilon_2)^2$}
    \label{fig:decompose}
    \end{figure}

    Take any $p \in \gen{(g_1, \varepsilon_1), \ldots, (g_n, \varepsilon_n)}_{\gp}$, we can write it as a product
    \[
    p = h_1 h_2 \cdots h_m,
    \]
    where $m \in \N$ and $h_i \in \{(g_1, \varepsilon_1), \ldots, (g_n, \varepsilon_n)\}^{\pm}$.
    We use induction on the length $m$ to show that $p$ can be written in the form~\eqref{eq:p_standard}.

    The base case $m=0$ is trivial.
    For the induction step, suppose we have written $h_1 \cdots h_{m-1}$ in the form~\eqref{eq:p_standard}:
    \[
    h_1 \cdots h_{m-1} = \li{(g_1, \varepsilon_1)^{z_{11}} \cdots (g_n, \varepsilon_n)^{z_{1n}}}c_1 \cdot \cdots \cdot \li{(g_1, \varepsilon_1)^{z_{s1}} \cdots (g_n, \varepsilon_n)^{z_{sn}}}c_s \cdot (g_1, \varepsilon_1)^{z_{01}} \cdots (g_n, \varepsilon_n)^{z_{0n}}.
    \]
    If $h_m = (g_n, \varepsilon_n)$ or $(g_n, \varepsilon_n)^{-1}$ then we simply increase or decrease $z_{0n}$ and obtain the same form for $h_1 \cdots h_{m-1} h_m$.
    Now suppose $h_m = (g_i, \varepsilon_i)$ with $1 \leq i \leq n-1$.
    The case $h_m = (g_i, \varepsilon_i)^{-1}$ is analogous.
    It suffices to show that the final part of the product $h_1 \cdots h_{m-1} h_m$, which is
    \[
    (g_1, \varepsilon_1)^{z_{01}} \cdots (g_n, \varepsilon_n)^{z_{0n}} (g_i, \varepsilon_i),
    \]
    can be written in the required form~\eqref{eq:p_standard}. 
    For brevity we now denote $a_i \coloneqq (g_i, \varepsilon_i)^{\sgn(z_{i0})}$ where $\sgn(z_{i0}) \in \{1, -1\}$ denotes the sign of $z_{i0}$, and let $x_{i0} \coloneqq |z_{i0}|$.
    So
    \[
    (g_1, \varepsilon_1)^{z_{01}} \cdots (g_n, \varepsilon_n)^{z_{0n}} (g_i, \varepsilon_i) = a_1^{x_{01}} \cdots a_n^{x_{0n}} a_i.
    \]
    Applying 
    \[
    a_j a_i = a_i a_j  [a_j^{-1}, a_i^{-1}]
    \]
    repeatedly, we move the final factor $a_i$ in $a_1^{x_{01}} \cdots a_n^{x_{0n}} a_i$ to the left until it reaches the position after
    $a_i^{x_{0i}}$, that is,
    \begin{align*}
        & a_1^{x_{01}} \cdots a_n^{x_{0n}} a_i \\
        =\; & a_1^{x_{01}} \cdots a_n^{x_{0n}-1} (a_n a_i) \\
        =\; & a_1^{x_{01}} \cdots a_n^{x_{0n}-1} a_i a_n [a_n^{-1}, a_i^{-1}] \\
        =\; & a_1^{x_{01}} \cdots a_n^{x_{0n}-2} (a_n a_i) a_n [a_n^{-1}, a_i^{-1}] \\
        =\; & a_1^{x_{01}} \cdots a_n^{x_{0n}-2} a_i a_n [a_n^{-1}, a_i^{-1}] a_n [a_n^{-1}, a_i^{-1}] \\
        & \vdots \\
        =\; & a_1^{x_{01}} \cdots a_i^{x_{0i}} \cdot a_i a_{i+1} [a_{i+1}^{-1}, a_i^{-1}] a_{i+1} [a_{i+1}^{-1}, a_i^{-1}] \cdots a_n [a_n^{-1}, a_i^{-1}] a_n [a_n^{-1}, a_i^{-1}].
    \end{align*}
    Then moving all the factors not in brackets to the right, the above expression becomes
    \[
    \li{a_1^{x_{01}} \cdots a_i^{x_{0i}+1} a_{i+1}}[a_{i+1}^{-1}, a_i^{-1}] \cdots \li{a_1^{x_{01}} \cdots a_i^{x_{0i}+1} \cdots a_n^{x_{0n}}}[a_n^{-1}, a_i^{-1}] \cdot \big(a_1^{x_{01}} \cdots a_i^{x_{0i}+1} \cdots a_n^{x_{0n}}\big),
    \]
    which is of the form
    \begin{equation*}
    \li{(g_1, \varepsilon_1)^{z_{11}} \cdots (g_n, \varepsilon_n)^{z_{1n}}}c_1 \cdot \li{(g_1, \varepsilon_1)^{z_{21}} \cdots (g_n, \varepsilon_n)^{z_{2n}}}c_2 \cdot \cdots \cdot \li{(g_1, \varepsilon_1)^{z_{s1}} \cdots (g_n, \varepsilon_n)^{z_{sn}}}c_s \cdot (g_1, \varepsilon_1)^{z_{01}} \cdots (g_n, \varepsilon_n)^{z_{0n}}
    \end{equation*}
    with $c_j \in C_{++} \cup C_{+-} \cup C_{-+} \cup C_{--}$.
    This finishes the induction on $m$ and completes Step 2. 

    \smallskip

    \textbf{Combining the two steps}, we first write
    \[
        f = \li{p_1}t_1 \cdot \li{p_2}t_2 \cdot \cdots \cdot \li{p_m}t_m \cdot p_0.
    \]
    Then, in each $\li{p_i}t_i, 1 \leq i \leq m$, we replace $p_i$ with the expression~\eqref{eq:p_standard}, and obtain
    \begin{align*}
    \li{p_i}t_i = p_i t_i p_i^{-1}
    = & \li{(g_1, \varepsilon_1)^{z_{11}} \cdots (g_n, \varepsilon_n)^{z_{1n}}}c_1 \cdots \li{(g_1, \varepsilon_1)^{z_{s1}} \cdots (g_n, \varepsilon_n)^{z_{sn}}}c_s \cdot (g_1, \varepsilon_1)^{z_{01}} \cdots (g_n, \varepsilon_n)^{z_{0n}} \cdot t_i \\
    & \quad \cdot \big((g_1, \varepsilon_1)^{z_{01}} \cdots (g_n, \varepsilon_n)^{z_{0n}}\big)^{-1} \cdot \li{(g_1, \varepsilon_1)^{z_{s1}} \cdots (g_n, \varepsilon_n)^{z_{sn}}}c_s^{-1} \cdots \li{(g_1, \varepsilon_1)^{z_{11}} \cdots (g_n, \varepsilon_n)^{z_{1n}}}c_1^{-1} \\
    = & \li{(g_1, \varepsilon_1)^{z_{11}} \cdots (g_n, \varepsilon_n)^{z_{1n}}}c_1 \cdots \li{(g_1, \varepsilon_1)^{z_{s1}} \cdots (g_n, \varepsilon_n)^{z_{sn}}}c_s \cdot \li{(g_1, \varepsilon_1)^{z_{01}} \cdots (g_n, \varepsilon_n)^{z_{0n}}} t_i \\
    & \quad \quad \cdot \li{(g_1, \varepsilon_1)^{z_{s1}} \cdots (g_n, \varepsilon_n)^{z_{sn}}}c_s^{-1} \cdots \li{(g_1, \varepsilon_1)^{z_{11}} \cdots (g_n, \varepsilon_n)^{z_{1n}}}c_1^{-1}.
    \end{align*}
    This is a product of the form~\eqref{eq:standard_form} in the statement of the proposition.
    Then, we replace the term $p_0$ with the expression~\eqref{eq:p_standard}, with the additional observation that $z_{01} = \cdots = z_{0n} = 0$ since $p_0 \in \GZn$.
    Together, this expresses $f$ in the required form~\eqref{eq:standard_form} and proves the proposition.
\end{proof}

We now reduce Subgroup Membership in $G \wr \Z^n$ to nested subgroup membership in $\GZn$ (see Section~\ref{subsec:nested}).

\begin{proposition}\label{prop:group_to_nested}
    Subgroup Membership in $G \wr \Z^n$ reduces to nested subgroup membership in $\GZn$.
\end{proposition}
\begin{proof}    
    By Lemma~\ref{lem:subgroup_special} we reduce to the case of determining whether $(f, 0) \in G \wr \Z^n$ is in the subgroup $H = \gen{(g_1, \varepsilon_1), \ldots, (g_n, \varepsilon_n), (f_1, 0), \ldots, (f_k, 0)}_{\gp}$.
    By Proposition~\ref{prop:standard_form}, elements of $H \cap \GZn$ are products of tiles of the form
    \begin{equation}\label{eq:single_tile}
    \li{(g_1, \varepsilon_1)^{z_1} \cdots (g_n, \varepsilon_n)^{z_n}}t,
    \end{equation}
    where $t \in T$ with
    \[
    T = \{f_1, \ldots, f_k\}^{\pm} \cup C_{++} \cup C_{+-} \cup C_{-+} \cup C_{--}.
    \]
    The overall idea of the proof is to show that the conjugated tiles $\li{(g_1, \varepsilon_1)^{z_1} \cdots (g_n, \varepsilon_n)^{z_n}}t,\; t \in T$, can have only finitely many patterns, depending on the sign and size of $z_1, \ldots, z_n$.
    See Figure~\ref{fig:reduce} for an illustration.

    \begin{figure}[h!]
        \centering
        \includegraphics[width=0.8\textwidth,height=0.8\textheight,keepaspectratio, trim={4.24cm 1.65cm 2.95cm 0.8cm},clip]{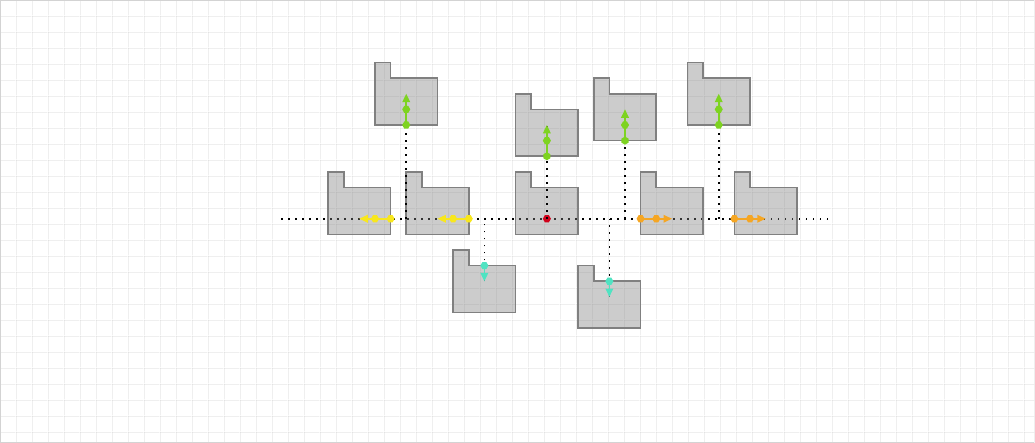}
        \caption{Different conjugations of the same tile, depending on its position (compare Figure~\ref{fig:nested}).}
        \label{fig:reduce}
    \end{figure}
    
    For a tile $b \in \GZn$, let $\|b\|$ denote the smallest number $y \in \N$ such that 
    \[
    \supp(b) \in [-y, y]^n.
    \]
    Let
    \begin{equation}\label{eq:equiv_def_d}
    d \coloneqq \max_{t \in T} \|t\| + \max_{1 \leq i \leq n} \|g_i\|.
    \end{equation}
    Define the following tile sets for $i = 1, \ldots, n$:
    \begin{align*}
    B_{i, +} & \coloneqq \left\{ \li{(g_i, \varepsilon_i)^d (g_{i+1}, \varepsilon_{i+1})^{z_{i+1}} \cdots (g_n, \varepsilon_n)^{z_n}}t \;\middle|\; |z_{i+1}| < d, \ldots, |z_n| < d,\;  t \in T \right\} \subset \GZn, \\
     B_{i, -} & \coloneqq \left\{ \li{(g_i, \varepsilon_i)^{-d} (g_{i+1}, \varepsilon_{i+1})^{z_{i+1}} \cdots (g_n, \varepsilon_n)^{z_n}}t \;\middle|\; |z_{i+1}| < d, \ldots, |z_n| < d,\;  t \in T \right\} \subset \GZn,
    \end{align*}
    and
    \[
    B_0 \coloneqq  \left\{ \li{(g_1, \varepsilon_1)^{z_1} \cdots (g_n, \varepsilon_n)^{z_n}}t \;\middle|\; |z_1| < d, \ldots, |z_n| < d,\;  t \in T \right\} \subset \GZn.
    \]
    We now show that $(f, 0) \in H \cap \GZn$ if and only if $f$ is in the nested subgroup
    \[
    \gen{B_0 \,\middle|\, B_{1, +}, B_{1, -} \,\middle|\, \cdots \,\middle|\, B_{n, +}, B_{n, -}} \leq \GZn.
    \]
    Fix a tile 
    \[
    \li{(g_1, \varepsilon_1)^{z_1} \cdots (g_n, \varepsilon_n)^{z_n}}t
    \]
    of the form~\eqref{eq:single_tile}.

    \smallskip
    \begin{adjustwidth}{3mm}{3mm}
    \textbf{Claim.} Let $i$ be the largest number with $0 \leq i \leq n$ such that $|z_{i+1}| < d, \ldots, |z_n| < d$.
    Then:
    \begin{enumerate}[nosep, label=(\roman*)]
        \item If $i \geq 1$ and $z_i \geq 0$, then $z_i \geq d$, and
        \[
        \li{(g_1, \varepsilon_1)^{z_1} \cdots (g_n, \varepsilon_n)^{z_n}}t = \li{(z_1, \ldots, z_i - d, 0, \ldots, 0)}{b}.
        \]
        where
        \[
        b \coloneqq \li{(g_i, \varepsilon_i)^{d} (g_{i+1}, \varepsilon_{i+1})^{z_{i+1}} \cdots (g_n, \varepsilon_n)^{z_n}}t \in B_{i, +}.
        \]
        \item If $i \geq 1$ and $z_i < 0$, then $z_i \leq -d$, and
        \[
        \li{(g_1, \varepsilon_1)^{z_1} \cdots (g_n, \varepsilon_n)^{z_n}}t = \li{(z_1, \ldots, z_i + d, 0, \ldots, 0)}{b}.
        \]
        where
        \[
        b \coloneqq \li{(g_i, \varepsilon_i)^{-d} (g_{i+1}, \varepsilon_{i+1})^{z_{i+1}} \cdots (g_n, \varepsilon_n)^{z_n}}t \in B_{i, -}.
        \]
        \item If $i = 0$, then $|z_1| < d, \ldots, |z_n| < d$, and
        \[
        \li{(g_1, \varepsilon_1)^{z_1} \cdots (g_n, \varepsilon_n)^{z_n}}t \in B_0.
        \]
    \end{enumerate}
    \smallskip
    
    \noindent$\blacktriangleright$
    It suffices to prove (i), since (ii) is analogous and (iii) is obvious.
    We can write each $(g, \varepsilon)$ as a product $g \varepsilon$ where $g \in \GZn \leq G \wr \Z^n$ and $\varepsilon = (e^{\Z^n}, \varepsilon) \in G \wr \Z^n$. 
    Then
    \[
    (g_1, \varepsilon_1)^{z_1} \cdots (g_i, \varepsilon_i)^{z_i - d} = (g_1 \varepsilon_1 \cdots g_1 \varepsilon_1) \cdots (g_i \varepsilon_i \cdots g_i \varepsilon_i) = g_1 \cdot \li{\varepsilon_1}g_1 \cdots \li{\varepsilon_1^{z_1} \cdots \varepsilon_i^{z_i - d - 1}}g_i \cdot \left(\varepsilon_1^{z_1} \cdots \varepsilon_i^{z_i - d}\right).
    \]
    (If $z_i = d$ then the term $\li{\varepsilon_1^{z_1} \cdots \varepsilon_i^{z_i - d - 1}}g_i$ should be replaced by $\li{\varepsilon_1^{z_1} \cdots \varepsilon_{i-1}^{z_{i-1} - 1}}g_i$).
    Consequently, letting
    \[
    b \coloneqq \li{(g_i, \varepsilon_i)^{d} (g_{i+1}, \varepsilon_{i+1})^{z_{i+1}} \cdots (g_n, \varepsilon_n)^{z_n}}t \in B_{i, +},
    \]
    we have
    \begin{align*}
        & \li{(g_1, \varepsilon_1)^{z_1} \cdots (g_n, \varepsilon_n)^{z_n}}t 
        = \li{(g_1, \varepsilon_1)^{z_1} \cdots (g_i, \varepsilon_i)^{z_i - d}}b \\
        =\; & g_1 \cdot \li{\varepsilon_1}g_1 \cdots \li{\varepsilon_1^{z_1} \cdots \varepsilon_i^{z_i - d - 1}}g_i \left(\varepsilon_1^{z_1} \cdots \varepsilon_i^{z_i - d}\right) b \left(\varepsilon_1^{z_1} \cdots \varepsilon_i^{z_i - d}\right)^{-1} \left(g_1 \cdot \li{\varepsilon_1}g_1 \cdots \li{\varepsilon_1^{z_1} \cdots \varepsilon_i^{z_i - d - 1}}g_i \right)^{-1} \\
        = & \li{u}{\left( \li{(z_1, \ldots, z_i - d, 0, \ldots, 0)}{b} \right)},
    \end{align*}
    where
    \[
    u \coloneqq g_1 \cdot \li{\varepsilon_1}g_1 \cdots \li{\varepsilon_1^{z_1} \cdots \varepsilon_i^{z_i - d - 1}}g_i \in \GZn.
    \]
    To prove the claim~(i) it suffices to show that $\li{u}{\left( \li{(z_1, \ldots, z_i - d, 0, \ldots, 0)}{b} \right)} = \li{(z_1, \ldots, z_i - d, 0, \ldots, 0)}{b}$.
    We do this by showing $\supp(u) \cap \supp\big(\li{(z_1, \ldots, z_i - d, 0, \ldots, 0)}b\big) = \emptyset$.
    
    Indeed,
    \begin{align*}
        & \supp\big(\li{(z_1, \ldots, z_i - d, 0, \ldots, 0)}b\big) \\
        =\;& (z_1, \ldots, z_i - d, 0, \ldots, 0) + \supp(b) \\
        =\;& (z_1, \ldots, z_i - d, 0, \ldots, 0) + \supp\Big(\li{(g_i, \varepsilon_i)^{d} (g_{i+1}, \varepsilon_{i+1})^{z_{i+1}} \cdots (g_n, \varepsilon_n)^{z_n}}t\Big) \\
        =\;& (z_1, \ldots, z_i - d, 0, \ldots, 0) + (0, \ldots, 0, d, z_{i+1} \ldots, z_n) + \supp(t) \\
        =\;& (z_1, \ldots, z_n) + \supp(t) \\
        \subseteq\;& \Z^{i-1} \times \big[z_i - \max_{t \in T} \|t\|,\, \infty \big) \times \Z^{n-i},
    \end{align*}
    while
    \begin{align*}
    \supp(u) & \subseteq \supp(g_1) \cup \supp\big(\li{\varepsilon_1}g_1\big) \cup \cdots \cup \supp \big(\li{\varepsilon_1^{z_1} \cdots \varepsilon_i^{z_i - d - 1}}g_i \big) \\
    & \subseteq \bigcup_{1 \leq j \leq i} \; \bigcup_{\substack{\alpha_1 \in \Z, \ldots, \alpha_{i-1} \in \Z, \\ 0 \leq \alpha_i \leq z_i - d - 1}} (\alpha_1, \ldots, \alpha_i, 0, \ldots, 0) + \supp(g_j) \\
    & \subseteq \Z^{i-1} \times \big(-\infty,\, z_i - d - 1 + \max_{1 \leq j \leq n} \|g_j\| \big] \times \Z^{n-i}.
    \end{align*}
    Therefore the choice of $d = \max_{t \in T} \|t\| + \max_{1 \leq i \leq n} \|g_i\|$ yields
    \[
    \supp(u) \cap \supp\big(\li{(z_1, \ldots, z_i - d, 0, \ldots, 0)}b\big) = \emptyset.
    \]
    Therefore $\li{u}{\left( \li{(z_1, \ldots, z_i - d, 0, \ldots, 0)}{b} \right)} = \li{(z_1, \ldots, z_i - d, 0, \ldots, 0)}{b}$.
    This combined with $\li{(g_1, \varepsilon_1)^{z_1} \cdots (g_n, \varepsilon_n)^{z_n}}t = \li{u}{\left( \li{(z_1, \ldots, z_i - d, 0, \ldots, 0)}{b} \right)}$ completes the proof of the claim.
    \hfill $\blacktriangleleft$
    \end{adjustwidth}
    \smallskip

    \noindent Coming back to the proof of the proposition, a tile is in $H \cap \GZn$ if and only if it is a product of tiles of the form
    \[
    \li{(g_1, \varepsilon_1)^{z_1} \cdots (g_n, \varepsilon_n)^{z_n}}t = \li{(z_1, \ldots, z_{i-1}, y_i, 0, \ldots, 0)}{b},
    \]
    where $z_1, \ldots, z_{i-1}, y_i, b$, fall in one of the following three cases of the claim:
    \begin{enumerate}[nosep, label=(\roman*)]
        \item $1 \leq i \leq n$, $z_1, \ldots, z_{i-1} \in \Z$, $y_i = z_i - d \in \N$, and $b \in B_{i, +}$,
        \item or, $1 \leq i \leq n$, $z_1, \ldots, z_{i-1} \in \Z$, $y_i = z_i + d \in -\N$, and $b \in B_{i, -}$,
        \item or, $i = 0$ and $b \in B_0$.
    \end{enumerate}
    \smallskip
    
    By the definition of nested subgroups (Section~\ref{subsec:nested}), products of such tiles are exactly elements of the nested subgroup
    $
    \gen{B_0 \,\middle|\, B_{1, +}, B_{1, -} \,\middle|\, \cdots \,\middle|\, B_{n, +}, B_{n, -}} \leq \GZn
    $.
    We conclude that $(f, 0) \in H \cap \GZn$ if and only if
    \[
    f \in \gen{B_0 \,\middle|\, B_{1, +}, B_{1, -} \,\middle|\, \cdots \,\middle|\, B_{n, +}, B_{n, -}}.
    \]
    Therefore, Subgroup Membership in $G \wr \Z^n$ reduces to nested subgroup membership in $\GZn$.
\end{proof}

From this we obtain the main result of this section:
\thmgroup*
\begin{proof}
    By Proposition~\ref{prop:group_to_nested}, Subgroup Membership in $G \wr \Z^n$ reduces to nested subgroup membership in $\GZn$.
    This is decidable by Theorem~\ref{thm:nested_tile}.
\end{proof}

\section{Conclusion and future work}

In this paper we defined a notion of standard basis for shift-stable subgroups of $\GNn$ for finite $G$, and used it to prove decidability of Subgroup Membership in $G \wr \Z^n$.
We discuss here some possible directions in which this work can be further developed:
\begin{enumerate}[nosep, label=(\arabic*)]
    \item One can work on replacing $G$ with infinite, non-abelian groups with suitable Noetherian properties.
    The most natural examples are \emph{virtually polycyclic} groups~\cite{baumslag1991algorithmic}, and possibly,  metabelian groups~\cite{baumslag1994algorithmic}.
    This may result in finding new classes of solvable groups with a decidable Subgroup Membership problem (cf.~\cite[Chapter~9]{LennoxRobinson2004}).
    Nevertheless, note that there are finitely presented 3-step solvable groups with an undecidable word problem (and hence an undecidable Subgroup Membership problem), by a result of Kharlampovich~\cite{kharlampovich1981finitely}.
    \item In parallel, one can work on replacing $\Z^n$ (and $\N^n$) with a non-commutative group (and its submonoid). This can be compared to the generalization from classical Gr\"{o}bner basis to non-commutative Gr\"{o}bner basis~\cite{Mora1994}.
    Natural candidates for this replacement are polycyclic groups and free groups.
    For free groups, Lohrey, Steinberg and Zetzsche~\cite{lohrey2015rational} showed that Rational Subset Membership is decidable in $G \wr F$, where $G$ is finite and $F$ is (virtually) free.
    It would be interesting to explore the connections between their techniques and the ones developed in this paper.
    \item It is also natural to consider other algorithmic problems in $G \wr \Z^n$ with finite $G$.
    Two prominent problems are \emph{Subgroup Intersection} (whether two finitely generated subgroups have trivial intersection) and \emph{Submonoid Membership} (whether a given element is in a given finitely generated submonoid).
    For Subgroup Intersection, Baumslag, Miller and Ostheimer~\cite{baumslag2010subgroups} showed its decidability in wreath products $G \wr \Z^n$ with finitely generated abelian $G$.
    To extend their result to finite non-abelian $G$ likely requires generalizing algorithms for \emph{ideal intersection}~\cite[Chapter~4.3]{CoxLittleOShea2015} to the stable subgroup setting.
    For Submonoid Membership, Dong~\cite{dong2025lamplighter} showed its decidability in $G \wr \Z^n$ where $G = \Z/p$ for prime $p$.
    The main argument in~\cite{dong2025lamplighter} is a reduction to \emph{S-unit equations} over fields of positive characteristic~\cite{adamczewski2012vanishing, derksen2012linear}.
    This result has recently been generalized to $G = \Z/(p^a q^b)$ with prime $p, q$~\cite{DongShafrir2026}.
    It would be interesting to consider the case where $G$ is a non-abelian finite $p$-group, and attempt to generalize the S-unit equation approach into non-abelian settings.
    Finally, we recall that by a result of Lohrey and Steinberg~\cite{lohrey2011tilings}, \emph{Rational Subset Membership} is undecidable in $G \wr \Z^n$ for non-trivial $G$ and $n \geq 2$.
\end{enumerate}

\section*{Acknowledgements} The author is supported by a Fellowship by Examination at Magdalen College, Oxford.
%Remove for submission

\bibliography{lampmember}

\end{document}